%% file: main.tex
\documentclass[12pt,a4paper]{report}

\input{opt}
\begin{document} 

\pagenumbering{gobble}
\pagenumbering{arabic}
\pagestyle{customoneside}

\begin{titlepage}
    \centering
    \vspace*{3cm}

    {\Huge\bfseries A Generalised Jordan Normal Form \\ and Its Computation Over Finite Fields\par}
    \vspace{2cm}

    {\Large Alia Bonnet\par}
    \vfill

    {\large
    Bachelor's Thesis in Mathematics\\[0.3cm]
    Submitted to the Faculty of Mathematics, Computer Science and Natural Sciences\\
    Chair of Algebra and Representation Theory\\
    RWTH Aachen University\par
    }
    \vspace{1.5cm}

    {\large
    \today\\[0.5cm]
    Supervisors:\\
    Professor Dr. Niemeyer\\
    Professor Dr. Horn
    }
\end{titlepage}

\chapter*{Abstract}
\begin{quote}
	We present a generalised version of the classical Jordan normal form, extending its applicability to matrices over arbitrary fields. This form is constructed by decomposing the underlying vector space into primary cyclic subspaces. We then introduce an efficient algorithm for computing the Jordan normal form over finite fields. Finally, we illustrate the utility of this approach by summarising several results from Franceschi (2020), where the Jordan normal form plays a central role in addressing the conjugacy problem in finite classical groups.
\end{quote}

\newpage

\tableofcontents

\newpage


\input{parts/00intro}

\input{parts/01pre}

\input{parts/02jnf}

\input{parts/04alg}

\input{parts/03special}

\input{parts/05conj}

\printbibliography
\end{document}

%% file: opt.tex
\usepackage[utf8]{inputenc}
\usepackage[T1]{fontenc}
\usepackage{lmodern}
\usepackage{yhmath}
\usepackage{ upgreek }
\usepackage{comment}
\usepackage{xcolor}
\usepackage{aligned-overset}
\usepackage[
backend=biber,
style=alphabetic,
]{biblatex}

\usepackage{tikz}
\usetikzlibrary{patterns}

\usepackage{graphicx}
\graphicspath{ {./pics/} }

\usepackage[labelsep=none,labelformat=empty]{caption}
\usepackage{kantlipsum}

\usepackage{csquotes}
\usepackage[toc,page]{appendix}
\usepackage[english]{babel}
\usepackage{float}
    \restylefloat{table}
\usepackage[
  bookmarksopen, 
  bookmarksnumbered, 
colorlinks = false,
  allbordercolors = cyan, 
  pdftitle={The generalised Jordan normal form},
  pdfauthor={Alia Bonnet},
  pdfstartview={FitH}, 
  ]{hyperref}

\usepackage{geometry}
\usepackage{titleps}
    \newpagestyle{customoneside}{%
    \sethead{}{}{}
    \setfoot{}{\thepage}{}
    }
    \settitlemarks{section,subsection,subsubsection}
\usepackage{setspace}

\usepackage{tocloft}
    
\usepackage[ruled,vlined,english]{algorithm2e}
\RestyleAlgo{ruled}
\SetKw{Continue}{continue}
\SetKwProg{Fn}{Function}{:}{end}
\DontPrintSemicolon 

\usepackage{amsmath}
\usepackage{amssymb}
\usepackage{amsfonts}
\usepackage{mathtools}
\usepackage{bbm}
\usepackage{dsfont}

\DeclarePairedDelimiter\matspan{\langle}{\rangle} 
\DeclarePairedDelimiter\ceil{\lceil}{\rceil}

\DeclarePairedDelimiter\erz{\langle}{\rangle}

\DeclarePairedDelimiter\abs{\lvert}{\rvert}

\DeclareMathOperator{\kernel}{Ker}
\DeclareMathOperator{\bild}{Im}

\DeclareMathOperator{\lcm}{lcm}
\DeclareMathOperator{\JNF}{JNF}

\newcommand{\GL}[2]{\operatorname{GL}(#1,#2)}
\newcommand{\Sp}[2]{\operatorname{Sp}(#1,#2)}
\newcommand{\U}[2]{\operatorname{U}(#1,#2)}
\newcommand{\SL}[2]{\operatorname{SL}(#1,#2)}
\newcommand{\SO}[3]{\operatorname{SO}^{#1}(#2,#3)}
\newcommand{\SU}[2]{\operatorname{SU}(#1,#2)}
\newcommand{\F}{\mathbb{F}}
\newcommand\On{\mathcal{O}}

    \usepackage{tabularx}

    \let\oldforall\forall
    \let\forall\undefined
    \DeclareMathOperator{\forall}{\,\oldforall\!}
    
    \let\oldexists\exists
    \let\exists\undefined
    \DeclareMathOperator{\exists}{\,\oldexists\!}

\definecolor{ColDefi}{cmyk}{0.23,0.05,0.05,0}
\definecolor{LUHblue}{cmyk}{1,0.49,0,0.38}
\definecolor{DefinColor}{RGB}{225, 237, 164} 
\definecolor{BemerkColor}{RGB}{235,243,234}  
\definecolor{BspColor}{RGB}{240,190,190}    
\definecolor{SatzColor}{RGB}{240,200,185}    
\definecolor{LemmaColor}{RGB}{210,230,180}   
\definecolor{KorColor}{RGB}{250,229,232}     
\definecolor{PropColor}{RGB}{245,235,180}    
\usepackage{aliascnt}
\usepackage{cleveref}

\newenvironment{proof}{
    {\textit{Proof:\;} }%
    }{%
    \hspace*{\fill} $\square$%
    } 

\newenvironment{citeproof}{
	{\textit{Proof:\;} }%
}{%
} 

\usepackage[]{mdframed}
    \mdfsetup{skipabove=5pt,skipbelow=5pt}

\usepackage{nicematrix}

\usepackage[amsmath,thmmarks]{ntheorem}

    {\theorembodyfont{\normalfont}
    \theoremstyle{break} 
    \newmdtheoremenv[
        middlelinewidth = 2 ,%
        roundcorner = 3 pt ,%
        innertopmargin = 6,
        skipabove = 6,
        topline=false,
        bottomline=false,
        rightline=false,
        leftline=false,
        backgroundcolor = DefinColor ,%
    	ntheorem,
        ]{defi}{Definition}[section]
    }

    {
    \theorembodyfont{\normalfont}
    \theoremstyle{break} 
    \newaliascnt{lem}{defi}
    \newmdtheoremenv[
        middlelinewidth = 2 ,%
        roundcorner = 3 pt ,%
        innertopmargin = 6,
        skipabove = 6,
        topline=false,
        bottomline=false,
        rightline=false,
        leftline=false,
        backgroundcolor = LemmaColor ,%
    	ntheorem,
        ]{lem}[lem]{Lemma} 
    	}
    \aliascntresetthe{lem}

    {
    \theorembodyfont{\normalfont}
    \theoremstyle{break} 
    \newaliascnt{theorem}{defi}
    \newmdtheoremenv[
        middlelinewidth = 2 ,%
        roundcorner = 3 pt ,%
        innertopmargin = 6,
        skipabove = 6,
        topline=false,
        bottomline=false,
        rightline=false,
        leftline=false,
        backgroundcolor = SatzColor ,%
    	ntheorem,
        ]{theorem}[theorem]{Theorem}
        }
    \aliascntresetthe{theorem}

    {
    \theorembodyfont{\normalfont}
    \theoremstyle{break} 
    \newaliascnt{kor}{defi}
    \newmdtheoremenv[
        middlelinewidth = 2 ,%
        roundcorner = 3 pt ,%
        innertopmargin = 6,
        skipabove = 6,
        topline=false,
        bottomline=false,
        rightline=false,
        leftline=false,
        backgroundcolor = KorColor ,%
        middlelinecolor = green!70!black ,%
    	ntheorem,
        ]{kor}[kor]{Corollary}
        }
    \aliascntresetthe{kor}

    {
    	\theorembodyfont{\normalfont}
    	\theoremstyle{break} 
    	\newaliascnt{prop}{defi}
    	\newmdtheoremenv[
    	middlelinewidth = 2 ,%
    	roundcorner = 3 pt ,%
    	innertopmargin = 6,
    	skipabove = 6,
    	topline=false,
    	bottomline=false,
    	rightline=false,
    	leftline=false,
    	backgroundcolor = PropColor ,%
    	middlelinecolor = green!70!black ,%
    	ntheorem,
    	]{prop}[prop]{Proposition}
    }
    \aliascntresetthe{prop}

    \newaliascnt{bem}{defi}
    {\theorembodyfont{\normalfont}
    \theoremstyle{break} 
    \newmdtheoremenv[
        middlelinewidth = 2 ,%
        roundcorner = 3 pt ,%
        innertopmargin = 6,
        skipabove = 6,
        topline=false,
        bottomline=false,
        rightline=false,
        leftline=false,
        backgroundcolor = BemerkColor ,%
        middlelinecolor = green!70!black ,%
    	ntheorem,
        ]{bem}[bem]{Remark}
    }
    \aliascntresetthe{bem}
    
    \newaliascnt{bsp}{defi}
    {\theorembodyfont{\normalfont}
    \theoremstyle{break} 
    \newmdtheoremenv[
        middlelinewidth = 2 ,%
        roundcorner = 3 pt ,%
        innertopmargin = 6,
        skipabove = 6,
        topline=false,
        bottomline=false,
        rightline=false,
        leftline=false,
        backgroundcolor = BspColor ,%
        middlelinecolor = green!70!black ,%
    	ntheorem,
        ]{bsp}[bsp]{Example}
    }
    \aliascntresetthe{bsp}

    \crefname{defi}{Definition}{Definitions}
    \crefname{lem}{Lemma}{Lemmata}
    \crefname{theorem}{Theorem}{Theorems}
    \crefname{prop}{Proposition}{Propositions}
    \crefname{kor}{Corollary}{Corollaries}
    \crefname{bem}{Remark}{Remarks}
    \crefname{bsp}{Example}{Examples}
   \crefname{equation}{Equation}{Equations}
   \crefname{algocf}{Algorithm}{Algorithms}
   \crefname{algorithm}{Algorithm}{Algorithms} 
   \crefname{proof}{Proof}{Proofs}

    \usepackage[shortlabels]{enumitem}
    \setlist[itemize]{nosep}
    \setlist[enumerate]{nosep}

%% file: parts/00intro.tex
\chapter{Introduction}

The question of matrix similarity is a classical one in linear algebra. For a field $\F$ and some positive integer $n\in \mathbb{N}$, one may consider the following problems:
\begin{enumerate}
	\item Given two matrices $A, B \in \GL{n}{\F}$, determine whether they are similar or not. 
	\item If they are similar, compute a conjugating matrix $X \in \GL{n}{\F}$. 
	\item List a representative for each conjugacy class of $\GL{n}{\F}$. 
\end{enumerate}
They can be readily solved by using normal forms. The most commonly studied forms are the rational canonical form (also known as the Frobenius normal form) and the Jordan normal form. The Jordan form, however, is traditionally defined only over algebraically closed fields such as $\mathbb{C}$. In this thesis, we aim to extend the notion of the Jordan normal form to arbitrary fields, following the approach outlined in \cite{Ple08}. Moreover, we provide practical algorithms for computing this generalized Jordan form, which we have implemented in \texttt{GAP} for finite fields.

The construction of the Jordan normal form relies on analyzing the action of a matrix $A \in \F^{n \times n}$ on the vector space $V = \F^n$. By decomposing $V$ into $A$-invariant subspaces, one obtains, in a sense, a corresponding decomposition of $A$ itself. The proofs in this thesis are expressed in terms of matrices, rather than modules, to reflect the computational approach used in practice.

%% file: parts/01pre.tex
\chapter{Background}
We begin by recounting a few basic facts from linear algebra, following the outlines of \cite{Har70},\cite{Hal78} and \cite{Hof71}. This serves to both as a presentation of the general concepts used in the rest of this paper and to fix notation.

For the entirety of this thesis, let $\F$ be a field, $n \in \mathbb{N}$ and $V$ be an $n$-dimensional vector space, unless stated otherwise. 

\section{The linear groups and conjugation}

\begin{defi}
	\begin{enumerate}
		\item We call the set of all invertible matrices in $\F^{n \times n}$ the \textbf{general linear group} and denote it by $\GL{n}{\F}$. Furthermore, the \textbf{special linear group} $\SL{n}{\F}$ is the subset of all elements of $\GL{n}{\F}$ having determinant 1.
		\item We denote the set of automorphisms on $V$ by $\mathrm{GL}(V)$.
	\end{enumerate}
\end{defi}

\begin{defi}
	Let $\alpha \in \mathrm{End}(V)$ be an endomorphism and $\mathcal{B} = (b_1,\dots,b_n)$ be a basis of $V$. We call the matrix with $\alpha(b_i)$ as its $i$-th row \textbf{the matrix representation of $\alpha$ in terms of $\mathcal{B}$} and denote it by $M_{\mathcal{B}}(\alpha)\in \F^{n\times n}$.
\end{defi}

\begin{bem}\label{groupiso}
	Since homomorphisms are fully determined by their action on basis vectors, we can identify each endomorphism $\alpha \in \mathrm{End}(V)$ as a matrix $M_{\mathcal{B}}(\alpha) \in \GL{n}{\F}$ by choosing a basis $\mathcal{B}$ of $V$. So:
	\[\mathrm{GL}(V) \simeq \GL{n}{\F}.\]
\end{bem}

For $A \in \F^{n\times n}$ and $B \in \GL{n}{\F}$ we write $A^B := B^{-1}AB$ for the conjugation.

\begin{defi}
	Let $A,B \in \F^{n \times n}$. We call $A$ and $B$ \textbf{conjugate in $\GL{n}{\F}$} or \textbf{similar} if there exists a matrix $C \in \GL{n}{\F}$ such that 
	$$A^C = B.$$
\end{defi}

Similar matrices correspond to the same linear mappings under different bases. Thus, studying similarity classes of matrices is equivalent to classifying linear operators up to a change of basis.

\begin{theorem}
	Let $A, B \in \F^{n\times n}$. Then the following are equivalent:
	\begin{enumerate}
		\item $A$ and $B$ are similar.
		\item There exist bases $\mathcal{B}, \mathcal{B}'$ and a linear mapping $\alpha \in \mathrm{GL}(V)$ such that $M_{\mathcal{B}}(\alpha) = A$ and $M_{\mathcal{B'}}(\alpha) = B$, where $M_{\mathcal{B}}(\alpha)$ denotes the matrix representation of $\alpha$ in the basis $\mathcal{B}$.
	\end{enumerate}
\end{theorem}

\begin{citeproof}
	See \cite[p. 94]{Hof71}.
\end{citeproof}

\begin{bem}
	Similarly, given a basis $\mathcal{B} = (b_1,\dots,b_n)$ of $V$, we may identify each element $e = \sum_{i=1}^{n}a_i b_i \in V$ with $a_1,\dots,a_n \in \F$, as a vector $v_{\mathcal{B}}(e) = (\alpha_1,\dots,\alpha_n)\in \F^n$. So 
	$$V \simeq \F^n.$$
	For this reason, we let $V = \F^n$ from here on out. 
	
	Applying an endomorphism to an element $e \in V$ then corresponds to a matrix multiplication, i.e. 
	$$v_{\mathcal{B}}(\alpha(e)) = v_{\mathcal{B}}(e)M_{\mathcal{B}}(\alpha).$$
\end{bem}

\begin{citeproof}
	See \cite[p. 84]{Hof71}.
\end{citeproof}

\section{Minimal and characteristic polynomials} 
Let $\F[X]$ be a polynomial ring over $\F$. 

Given a polynomial $p \in \F[X]$ and a matrix $A \in \F^{n\times n}$, we may define $p(A)$ in the following way. 

\begin{bem}
	Let $A \in \F^{n \times n}$ and $p = p(X) = a_0X^0 + \dots + a_dX^d \in \F[X]$ be a polynomial. By defining $p(A) := a_0A^0 + \dots + a_dA^d$, the \textbf{evaluation mapping}
	$$\varepsilon_A: \F[X] \rightarrow \F^{n \times n}: p \mapsto p(A)$$
	is a homomorphism of rings.  
\end{bem}

\begin{proof}
	We have $\varepsilon_A(p + q) = (p+q)(A) = p(A) + q(A) = \varepsilon_A(p) + \varepsilon_A(q)$ and $\varepsilon_A(rp) = (rp)(A) = r(p(A)) = r\varepsilon_A(p)$ for all $p,q \in \F[X]$ and $r \in \F$. 
\end{proof}

\begin{prop}\label{minpol}
	Let $A \in \F^{n \times n}\backslash\{0\}$. Then the kernel of the evaluation mapping $$\kernel(\varepsilon_A)= \{p \in \F[X]: p(A) = 0\}$$
	is generated by a unique monic polynomial $p \in \F[X]$. 
\end{prop}

\begin{proof}
	$\kernel(\varepsilon_A)$ is a proper ideal of $\F[X]$. Since $\F$ is a field, $\F[X]$ is a principal ideal domain. Therefore any ideal of $\F[X]$ is generated by a single element $p$ in $\F[X]$, which is unique up to a unit in $\F$.
\end{proof}
	
\begin{defi}
	We call the monic polynomial $q$ generating $\kernel(\varepsilon_A)$, described in \cref{minpol} the \textbf{minimal polynomial of} $A$ and denote it by $\mu_A := q$.
\end{defi}

We note another characterization of the minimal polynomial, which warrants its naming.

\begin{theorem}
	Let $A \in \F^{n \times n}$. Then the minimal polynomial $\mu_A \in \F[X]$ of $A$ is uniquely determined by the following three properties:
	\begin{enumerate}
		\item $\mu_A$ is monic.
		\item $\mu_A(A) = 0$.
		\item $\mu_A|p$ for all $p \in \kernel(\varepsilon_A)$.
	\end{enumerate}
\end{theorem}

\begin{proof}
	All three properties follow directly from the definition of $\mu_A$. Now let $q \in \F[X]$ such that all three points hold. Then $q|\mu_A$. But $q(A) = 0$ so $q \in \kernel(\varepsilon_A)$ and as such $\mu_A|q$. Since $q$ is also monic we have $\mu_A = q$.
\end{proof}

Similarly, we may also consider minimal polynomials for vectors. 

\begin{defi}
	Let $v \in V\backslash\{0\}$ and $A \in \F^{n\times n}$. 
	If $p \in \F[X]$ is monic such that \begin{enumerate}
		\item $vp(A) = 0$ and 
		\item $p\mid q$ for all $q \in \F[X]$ with $vq(A) = 0$,
	\end{enumerate}
	we call $p$ the \textbf{minimal polynomial of $v$ with respect to $A$} and denote it by $\mu_{A,v} := p$.
\end{defi}

The definition above is well-defined: since there exists a minimal $k \in \mathbb{N}$, $k \leq n$, such that $(v,vA, \dots, vA^{k-1})$ is linearly dependent, there exists a unique $\F$-linear dependency such that $(a_0,\dots,a_k)$ with $a_k = 1$ and $a_0v + \dots + a_kvA^k = 0$.

\begin{prop}
	Let $A \in \F^{n\times n}$. Then for any $v \in V$ we have
	$$\mu_{A,v}|\mu_A,$$
	i.e., the minimal polynomial of $v$ divides the minimal polynomial of $A$. 
\end{prop}

\begin{proof}
	Since $v\mu_A(A) = v\cdot0 = 0$ we have $\mu_{A,v}|\mu_A$. 
\end{proof}

\begin{defi}
	Let $A \in \F^{n\times n}$. We call
		$$\chi_A(X) := \mathrm{det}(I - XA) \in \F[X]$$
	the \textbf{characteristic polynomial} of $A$. 
\end{defi}

\begin{theorem}[Cayley-Hamilton]
	Let $A \in \F^{n\times n}$. Then
	$$\chi_A(A) = 0$$
	i.e., the minimal polynomial of $A$ divides the characteristic polynomial of $A$. 
\end{theorem}

\begin{citeproof}
	See \cite[pp. 202-204]{Hof71}.
\end{citeproof}

\begin{prop}\label{samefactors}
	Let $A \in \F^{n\times n}$. Then $\chi_A$ and $\mu_A$ share the same distinct irreducible factors, though possibly with differing multiplicities. 
\end{prop}

\begin{citeproof}
	See \cite[p. 184]{Har70}.
\end{citeproof}

\begin{kor}\label{simpol}
	Let $A, B \in \F^{n\times n}$ be two similar matrices. Then 
	$$\mu_A = \mu_B \text{ and } \chi_A = \chi_B,$$
	i.e. $A$ and $B$ share the same minimal and characteristic polynomials.
\end{kor}

\begin{citeproof}
	See \cite[p. 192]{Hof71}.
\end{citeproof}

When checking two matrices for similarity, comparing their minimal and characteristic polynomials is an useful heuristic. However, it may not always be the case that matrices which share both their minimal and characteristic polynomials are already similar. Consider the following counter example:

\begin{bsp}
	Let $A,B \in \F_3^{4 \times 4}$ be two matrices with
	\[
	A:=
	\begin{pmatrix}
		. & . & 1 & .\\
		. & . & . & .\\
		. & . & . & .\\
		. & . & . & .	
	\end{pmatrix}\text{, } 
	B:=
	\begin{pmatrix}
		. & 1 & . & .\\
		. & . & . & .\\
		. & . & . & 1\\
		. & . & . & .		
	\end{pmatrix}.
	\]
	Then $\mu_A(X) = \mu_B(X) = X^2$ and $\chi_A(X) = \chi_B(X) = X^4$.
	But $\mathrm{Rank}(A) = 1 \neq 2 = \mathrm{Rank}(B)$, so $A$ and $B$ cannot be similar.
\end{bsp}

We will see later on however, that there exists a subclass of matrices for which similarity is equivalent to having the same minimal polynomial. 

Just as conjugate matrices represent the same linear transformation with respect to different bases, we can also reinterpret vectors in new coordinates by applying the change-of-basis matrix. If both the matrix and the vector are expressed relative to the same new basis, then the associated minimal polynomial remains unchanged.

\begin{kor}
	Let $A \in \F^{n\times n}$, $\mathcal{B} \in \GL{n}{\F}$ and $v \in V$. Then
	$$\mu_{A,v} = \mu_{A^\mathcal{B}, v\mathcal{B}}.$$
\end{kor}

\begin{proof}
	On the one hand we have
	\begin{align*}
		v\mathcal{B}\mu_{A,v}(A^{\mathcal{B}}) &= v\mathcal{B}\mathcal{B}^{-1}\mu_{A,v}(A)\mathcal{B}\\
		&= v\mu_{A,v}(A)\mathcal{B}\\
		&= 0
	\end{align*}
	so $\mu_{A,v} \mid \mu_{A^{\mathcal{B}},v\mathcal{B}}$. But since we can follow the same argument to show that $v\mu_{A^{\mathcal{B}},v\mathcal{B}}(A) = 0$ we also have that $\mu_{A^{\mathcal{B}},v\mathcal{B}} \mid \mu_{A,v}$. Hence we conclude $\mu_{A^{\mathcal{B}},v\mathcal{B}} = \mu_{A,v}$.
	
\end{proof}

Every polynomial $p \in \F[X]$ appears as the minimal polynomial and as the characteristic polynomial of a matrix. 

\begin{defi}
	Let $p \in \F[X]$ be a polynomial with $p(X) = X^n + c_{n-1}X^{n-1} ... + c_1X + c_0$,  and let $I_{n-1}$ denote the identity matrix of dimension $(n-1) \times (n-1)$. We call   
	$$M_p = \begin{pNiceArray}{c|ccc}
		0 & \Block{3-3}<\Large>{\mathbf{I_{n-1}}}\\
		\vdots \\
		0 \\
		\hline
		-c_0 & -c_1 & \dots & -c_{n-1}
	\end{pNiceArray} \in \F^{n\times n}$$
	the \textbf{companion matrix of $p$}.
\end{defi}

\begin{prop}
	For any polynomial $p \in \F[X]$ we have $\chi_{M_p} = \mu_{M_p} = p$.
\end{prop}

\begin{citeproof}
	See \cite[p. 230]{Hof71}.
\end{citeproof}

\chapter{Invariant Subspaces}
The construction of the Jordan normal form relies on decomposing the vector space $V$ into invariant subspaces. This chapter is dedicated to discussing the properties that make such subspaces so useful. We then turn to a particularly important class, the cyclic subspaces.

\section{A-invariance}

\begin{defi}
	Let $M$ be an abelian group and $R$ a ring, together with map $M \times R \rightarrow M, (m,r) \mapsto mr$. We call $M$ a \textbf{module over $R$} (or \textbf{$R$-module}) if the following conditions are satisfied for all $r,r_1,r_2 \in R$ and $m,m_1,m_2 \in M$: 
	\begin{enumerate}
		\item $(m_1+m_2)r = m_1r+ m_2r$,
		\item $m(r_1+r_2) = mr_1 + mr_2$,
		\item $m(r_1r_2) = (mr_1)r_2$,
		\item $m\cdot1 = m$.
	\end{enumerate} 
	If $M,N$ are two $R$-modules, we call a function $\varphi: M \rightarrow N$ a \textbf{module homomorphism} if 
	\begin{enumerate}
		\item $\varphi(m_1 + m_2) = \varphi(m_1) + \varphi(m_2)$,
		\item $\varphi(mr) = \varphi(m)r$,
	\end{enumerate}
	for all $r \in R$ and $m,m_1,m_2 \in M$. 
	Furthermore, if $\varphi$ is surjective, injective or bijective, we call it an \textbf{epimorphism}, \textbf{monomorphism} or \textbf{isomorphism}, respectively. 
\end{defi}

\begin{bem}\label{famodule}
	Let $A \in \F^{n \times n}$. We may consider $V$ as a canonical $\F[X]$ module via $A$ using the following map:
	$$V \times \F[X] \rightarrow V,\; (v,p) \mapsto v\cdot_A p := v\varepsilon_A(p) = vp(A).$$
	We may switch between looking at the repeated operation of $A$ on $V$ as an $\F$-vector space and $V$ as a $\F[X]$-module interchangeably.
\end{bem}

\begin{citeproof}
	See \cite[pp. 172-173]{Har70}.
\end{citeproof}

Note that since the polynomial ring $\F[X]$ is commutative, i.e. $p\cdot q = q\cdot p$ for all polynomials $p,q \in \F[X]$, we have 
$$vp(A)q(A) = vq(A)p(A)$$ 
for all vectors $v \in V$ and all matrices $A \in \F^{n\times n}$.

Considering $V$ as an $\F[X]$ module allows us to compare the actions of different matrices in $\F^{n \times n}$ on $V$ in a more precise way. We have seen before that similar matrices represent the same linear map under different bases. The following proposition formulates this in the language of modules. 

\begin{prop}
	 Let $A,B \in \F^{n\times n}$ be two similar matrices. Then $V$ as an $\F[X]$-module via $A$ is isomorphic to $V$ as an $\F[X]$-module via $B$. 
\end{prop}

\begin{proof}
	Since $A$ and $B$ are similar, there exists a matrix $C \in \GL{n}{\F}$ such that $C^{-1}AC = B$. Consider the following mapping: 
	\[\varphi: V \rightarrow V, v \mapsto vC.\]
	For $v,w \in V$ have $\varphi(v + w) = (v+w)C = vC + wC = \varphi(v) + \varphi(w)$ and for $p \in \F[X]$ we have 
	\[
	\begin{aligned}
		\varphi(v\cdot_A p) &=  vp(A)C \\
							&= vCC^{-1}p(A)C \\
							&= vCp(C^{-1}AC) \\
							&= vCp(B) \\
							&= \varphi(v)\cdot_B p.
	\end{aligned}
	\]
	So $\varphi$ is a module homomorphism. 
	
	Furthermore, since $C \in \GL{n}{\F}$, $\varphi$ is also bijective and thus a module isomorphism. 

\end{proof}

\begin{defi}
	Let $W \leq V$ be a subspace of $V$ and $A \in \F^{n\times n}$. We call $W$ an \textbf{$A$-invariant} subspace of $V$, and write $W \leq_A V$, if $W\cdotp A = \{wA\mid w \in W\} \subseteq W$.
\end{defi}

Only when a subspace is $A$-invariant it is well defined to consider the action of $A$ on that subspace. 

\begin{bem}
	Let $A \in \F^{n\times n}$. Furthermore, let $W \leq_A V$ be an  $A$-invariant subspace of $V$.
	\begin{enumerate}
	\item The mapping $W \times \F[X] \rightarrow W, (w,p) \mapsto wp(A)$ is well defined, so we may consider $W$ as an $\F[X]$-submodule of $V$.
	\item Similarly, for $v + W \in V/W$, we have 
	$$(v+W)A = vA + WA = vA + W \in V/W.$$
	So we may also consider $V/W$ as an $\F[X]$-module via $A$. 
	\end{enumerate}
\end{bem}

We translate the notion of $A$ acting on $A$-invariant subspaces of $V$ back into the language of matrices in the following proposition.

\begin{prop}\label{restriction}
	Let $A \in \F^{n\times n}$. Furthermore let $W \leq_A V$ be an $A$-invariant $s$-dimensional subspace of $V$ for $s \leq n$ and $\mathcal{B} = \{b_1,\dots,b_s\}$ a basis of $W$. If we extend $\mathcal{B}$ to a basis $\mathcal{B}' = \{b_1, \dots, b_s, b_{s+1}, \dots,b_n\}$ of $V$, then $A^{\mathcal{B}'^{-1}}$ has the block-triangular form
	$$A^{\mathcal{B}'^{-1}} = 
	\begin{pNiceArray}{c|c}
		A|_{W,\mathcal{B}} & 0^{s\times(n-s)}\\
		\hline
		C & A_{V/W,\mathcal{B}}
	\end{pNiceArray} \in \F^{n\times n}
	$$
	with $A|_{W,\mathcal{B}} \in \F^{s\times s}, A_{V/W, \mathcal{B}} \in \F^{(n-s) \times (n-s)}, C \in \F^{(n-s) \times s}$.
\end{prop}

\begin{citeproof}
	See \cite[p. 200]{Hof71}.
\end{citeproof}

\begin{bem}\label{restrictionop}
	\begin{enumerate}
		\item Given a map $\alpha \in \mathrm{End}(V)$, and an $\alpha$-invariant subpace $W$ of $V$, it is clear that the restriction
		$$\alpha|_W: W \rightarrow W, w \mapsto \alpha(w)$$ 
		is well defined. 
		By considering $\mathcal{B},\mathcal{B}_1,\mathcal{B}_2 := \{b_{s+1},\dots,b_n\}$ as described in \cref{restriction} and $A := M_{\mathcal{B}}(\alpha)$, we have that
		\[A|_{W,\mathcal{B}} = M_{\mathcal{B}_1}(\alpha|_W),\]
		i.e. $A|_{W,\mathcal{B}}$ is precisely the matrix representation of the induced endomorphism $\alpha|_W$ on $W$ in terms of $\mathcal{B}_1$.
		\item Similarly $\alpha$ induces a well-defined endomorphism on $V/W$
		\begin{align*}
			\alpha_{V/W}: V/W \rightarrow V/W, v+W \mapsto \alpha(v+W) &= \alpha(v) + \alpha(W)\\
			&= \alpha(v) + W.
		\end{align*}
		If we consider the basis $\mathcal{B}_{V/W}$ of $V/W$ induced by $\mathcal{B}$ as described in the canonical way, have that 
		$$A_{V/W,\mathcal{B}} = M_{\mathcal{B}_{V/W}}(\alpha_{V/W}).$$ 
		So $A_{V/W,\mathcal{B}}$ is the matrix representation of the induced endomorphism $\alpha_{V/W}$ on $V/W$ in terms of $\mathcal{B}_{V/W}$.
	\end{enumerate}
\end{bem}

\begin{citeproof}
	\begin{enumerate}
		\item See \cite[p. 200]{Hof71}.
		\item See \cite[p. 88]{Hal78}. 
	\end{enumerate}
\end{citeproof}

This motivates the following definition:

\begin{defi}
	Consider $A|_{W,\mathcal{B}}, A|_{V/W, \mathcal{B}}$ as defined in \cref{restriction}. We call $A|_{W,\mathcal{B}}$ the \textbf{matrix restriction} of $A$ to $W$ and $A_{V/W, \mathcal{B}}$ the \textbf{induced matrix} on $V/W$ in terms of $\mathcal{B}_W$. 
	
	For two different bases $\mathcal{B}, \mathcal{B}'$ of $W$ the matrix restrictions $A|_{W,\mathcal{B}}, A|_{W,\mathcal{B}'}$ and the induced matrices $A_{V/W, \mathcal{B}}, A_{V/W, \mathcal{B'}}$ are similar, so we may neglect specifying the basis and just write $A|_W$ and $A|_{V/W}$. 
\end{defi}

Since similar matrices represent the same linear map in different bases, it is natural to expect that their corresponding $\F[X]$-modules are essentially the same. The next proposition makes this precise.

\begin{kor}\label{moduleiso}
	Let $A\in \GL{n}{\F}$ and let $W \leq_A V$ be an $s$-dimensional subspace of $V$. Then $W$ as an $\F[X]$-module via $A$ is isomorphic to $\F^s$ as an $\F[X]$-module via $A|_{W}$.
\end{kor}

\begin{proof}
	This follows directly from \cref{restrictionop} by choosing a basis $\mathcal{B}$ of $V$ and considering $\alpha \in \mathrm{GL}(V)$ such that $A = M_{\mathcal{B}}(\alpha)$.
\end{proof}

An important property of invariant subspaces is that restrictions preserve divisibility relations between minimal polynomials.

\begin{prop}\label{restrictionpol}
	Let $A \in \F^{n\times n}$ and $W \leq V$ be an $A$-invariant subspace of $V$. Then 
	\[\mu_{A|_W} | \mu_A,\]
	i.e. the minimal polynomial of the matrix restriction divides the minimal polynomial of the matrix. 
\end{prop}

\begin{proof}
	Let $\mathcal{B}$ be a basis of $V$ as described in \cref{restriction}. Then $\mu_A = \mu_{A^{\mathcal{B}^{-1}}}$, so $\mu_A(A^{{\mathcal{B}}^{-1}}) = 0$. By the definition of matrix multiplication we then have $\mu_A(A|_{W}) = 0$, so we have $\mu_{A|_W}|\mu_A$. 
\end{proof}

\begin{theorem}\label{directsummat}
	Let $A \in \F^{n\times n}$ and $V = V_1 \oplus \dots \oplus V_k$ such that each $V_i$ is $A$-invariant. 
	If $\mathcal{B}_{i}$ is a basis of $V_i$ for $i \in \{1,\dots,k\}$, $\mathcal{B} := (\mathcal{B}_1,\dots,\mathcal{B}_k)$ is a basis of $V$. We can identify $\mathcal{B}$ as an element of $\GL{n}{\F}$ by writing the basis vectors as the rows of a matrix. Then $A^{\mathcal{B}^{-1}}$ is in the following block-diagonalform: 
	$$A^{\mathcal{B}^{-1}} = \begin{pNiceArray}{cccc}
		A|_{V_1} & 0 &\dots & 0\\
		0 & A|_{V_2} & & 0\\
		\vdots & & \ddots & \vdots\\
		0 & \dots & 0 & A|_{V_k}
	\end{pNiceArray}.$$ 
\end{theorem}

\begin{citeproof}
	See \cite[pp. 169-170]{Har70}.
\end{citeproof}

\begin{defi}
	Let $A \in \F^{n \times r}$ and $B \in \F^{s \times t}$ for $n,r,s,t \in \mathbb{N}$. We call the matrix
	$$A \oplus B :=
	\begin{pNiceArray}{c|c}
		A & 0^{n \times t}\\
		\hline
		0^{s \times r} & B
	\end{pNiceArray} \in \F^{n+s\times r+t},$$
	the \textbf{direct sum} of $A$ and $B$.
\end{defi}

This notation is motivated by \cref{directsummat}. Given the same setting as above, we see that $A^{\mathcal{B}^{-1}} = \bigoplus_{i=1}^k A|_{V_i}$. 

So given a matrix $A \in \F^{n\times n}$, any decomposition of $V$ into $A$-invariant direct sums also yields a basis $\mathcal{B}$ such that $A^{\mathcal{B}^{-1}}$ is the direct sum of smaller matrices. 
As such, we may refer to a decomposition of $V$ into $A$-invariant subspaces as simply a decomposition of $A$. 

\begin{prop}
	Let $A,B \in \F^{n \times n}$. Then $\mu_{A\oplus B} = \lcm(\mu_A,\mu_B)$.
\end{prop}

\begin{proof}
	For any polynomial $p \in \F[X]$ we have $p(A \oplus B) = p(A) \oplus p(B)$.
	
	Now let $\mu := \mu_{A \oplus B}$. It follows that: $0 = \mu(A \oplus B) = \mu(A) \oplus \mu(B)$ so $\mu(A)$ and $\mu(B)$ must both be $0$ and $\mu_A\mid \mu$, $\mu_B\mid \mu$. 
	
	Conversely, if $\mu_A \mid \mu$ and $\mu_B \mid \mu$, then $\mu(A) = \mu(B) = 0$ so $\mu(A \oplus B) = 0$. Since $\mu$ is minimal, we have that $\mu_{A\oplus B} = \lcm(\mu_A,\mu_B)$.
\end{proof}

\section{Cyclic subspaces}
One of the most powerful concepts we will use over the course of this paper is the span of vectors under matrices. It may seem to be a very simple one at first but it will prove to be very useful, especially in the practical computation of invariant subspaces.

\begin{defi}
	Let $v \in V$. Given a matrix $A \in \F^{n\times n}$, we call the set 
	$$\matspan{v}_A := v\cdot_A\mathbb{F}[X] = \{vp(A) \mid p\in \mathbb{F}[X]\}$$
	\textbf{the span of $v$ under $A$}. 
\end{defi}

We prove a few useful properties of the $A$-span of vectors. 

\begin{bem}\label{randomfacts}
	Let $v \in V$, $A \in \F^{n\times n}$ and $d:=\deg(\mu_{A,v})$. Then
	\begin{enumerate}
		\item $\matspan{v}_A$ is $A$-invariant. 
		\item $(v, vA, \dots, vA^{d-1})$ is a basis of $\matspan{v}_A$.
		\item $\dim(\matspan{v}_A) = d$. 
		\item $\mu_{A|_{\matspan{v}_A}} = \mu_{A,v}$
	\end{enumerate}
\end{bem}

\begin{proof}
	Let $w = vp(A) \in \matspan{v}_A$ for some polynomial $p(X) = \sum_{i=0}^ka_iX^i \in \F[X]$, $k \in \mathbb{N}$.
	\begin{enumerate}
		\item For $q(X) := X \in \F[X]$ we have:
		\[wA = vp(A)A = v\underbrace{(pq)}_{\in \F[X]}(A) \in \matspan{v}_A.\]
		\item If $k\geq d$ there exist polynomials $s,r \in \F[X]$ with $\deg(r) < d$ such that $p = \mu_{A,v}s + r$. So 
		\[w = vp(A) = v(\mu_{A,v}s+r)(A) = \underbrace{v\mu_{A,v}(A)}_{=0}s(A) + vr(A) = vr(A) \]
		So without loss of generality, let $k < d$. Then
		\[w = vp(A) = v\sum_{i=0}^{k}a_iA^i = \sum_{i=0}^{k}a_ivA^i\in \matspan{v,vA,\dots,vA^{d-1}}.\] 
		Conversely, for $\sum_{i=0}^{d-1}b_ivA^i \in \matspan{v,vA,\dots,vA^{d-1}}$ we can set the polynomial $q(X) := \sum_{i=0}^{d-1}b_iX^i$ and thus
		\[\sum_{i=0}^{d-1}b_ivA^i = v\sum_{i=0}^{d-1}b_iA^i= vq(A) \in \matspan{v}_A.\]
		
		Now Suppose that $\{v, vA, \dots, vA^{d-1}\}$ is linearly dependent, i.e. there exist $b_0,\dots,b_{d-1} \in \F$ such that $\sum_{i=0}^{d-1} b_ivA^i = 0$. But by setting $q(X) := \sum_{i=0}^{d-1}b_iX^i$ we then get $vq(A) = 0$. Since $\deg(q) = d-1 < \deg(\mu_{A,v})$, this is a contradiction to $\mu_{A,v}$ being the minimal polynomial of $v$. So $(v,vA,\dots,vA^{d-1})$ is a basis of $\matspan{v}_A$.
		\item Follows directly from the second statement. 
		\item See \cite[p.228]{Hof71}.
	\end{enumerate}
\end{proof}

\begin{defi}
	Let $A \in \F^{n \times n}$.
	\begin{enumerate}
	\item We call an $s$-dimensional subspace $W \leq V$ \textbf{A-cyclic}, if there exists a vector $w \in W$ such that $\matspan{w}_A = W$. We call $w$ a \textbf{cyclic vector for $W$}.
	\item We say that $A$ is \textbf{cyclic} if $V$ is $A$-cyclic. We call the vector $v$ such that $\matspan{v}_A = V$ a \textbf{cyclic vector for $A$}.
	\end{enumerate}
\end{defi}

Note the subtle but important difference between the definitions. The first definition is a statement about a subspace of $V$ in relation to $A$, while the second definition is a statement about the matrix $A$ itself. We may also consider cyclicity from the viewpoint of modules.

\begin{prop}
	Let $A \in \F^{n\times n}$ and $W \leq V$. Then $W$ is $A$-cyclic if and only if $W$ as an $\F[X]$-module via $A$ is generated by a single element $w \in W$. In that case, we may call $W$ a \textbf{cyclic $\F[X]$-module via $A$}.
\end{prop}

\begin{proof}
	Let $s := \dim(W)$. If $W$ is $A$-cyclic, there exists a $w \in W$ such that $\mathcal{W} = (w,wA,\dots,wA^{d-1})$ is a basis of $W$. Since we know from \cref{randomfacts} that $\mathcal{W}$ is also a basis of $\matspan{w}_A$, we have that $W = \matspan{w}_A = \{wp(A)|p \in \F[X]\}$, so $W$ as an $\F[X]$-module via $A$ is generated by $w$. 
	
	Conversely, if $W$ as an $\F[X]$-module via $A$ is generated by a single element, there exists a $w \in W$ such that for all $w' \in W$ there exists a $p \in \F[X]$ such that $w' = wp(A)$. So $\matspan{w}_A = W$, hence $(w,wA,\dots,wA^{s-1})$ is a basis of $W$.
\end{proof}

We see that a subspace $W$ being $A$-cyclic corresponds to the restriction $A|_W$ being a cyclic matrix. 

\begin{kor}\label{cyclicrestriction}
	Let $W \leq_A V$ be an $s$-dimensional, $A$-invariant subspace of $V$. Then the following are equivalent:
	\begin{enumerate}
		\item $A|_W$ is cyclic. 
		\item $W$ is $A$-cyclic.
	\end{enumerate}
\end{kor}

\begin{proof}
	Let $A|_W$ be cyclic, so $\F^s$ is a cyclic $\F[X]$-module via $A|_W$. From \cref{moduleiso} we know that $\F^s$ as an $\F[X]$-module via  $A|_W$ is isomorphic to $W$ as an $\F[X]$-module via $A$. So this is equivalent to $W$ also being a cyclic $\F[X]$-module via $A$ and thus $A$-cyclic.  
\end{proof}

What this tells us is that a decomposition of $V$ into $A$-cyclic subspaces corresponds to a decomposition of $A$ into cyclic matrices and vice versa.

Furthermore, $A$-invariant subspaces of $A$-cyclic spaces are also $A$-cyclic. 

\begin{lem}\label{cyclicsub}
	Let $A \in \F^{n \times n}$ and $W \leq V$ be $A$-cyclic. Furthermore let $U \leq_A W$. Then $U$ is $A$-cyclic.	
\end{lem}

\begin{proof}
	Let $t:=\dim(U)$ and consider a basis $(u_1,\dots,u_t)$ of $U$.
	Since $W$ is $A$-cyclic there exists a $w \in W$ such that $\matspan{w}_A = W$. So $(u_1,\dots,u_t) = (wp_1(A),\dots,wp_t(A))$ for some polynomials $p_1,\dots,p_t \in \F[X]$. 
	
	Because $U$ is $A$-invariant we have that $wp_i(A)q(A) \in U$ for all $i\in \{1\dots t\}$ and all $q \in \F[X]$. So $U = \matspan{wp_1(A),\dots,wp_t(A)} = \{wp_1(A)q_1(A),\dots,wp_t(A)q_t(A)\mid q_1,\dots,q_t \in \F[X]\}$.
	
	Now consider $f := \mathrm{gcd}(p_1,\dots,p_t)$. Then $U = \{wp_1(A)q_1(A),\dots,wp_t(A)q_t(A)\mid q_1,\dots,q_t \in \F[X]\} = \matspan{wf(A)q(A)\mid q\in \F[X]} = \matspan{wf(A)}_A$ so $wf(A)$ is a cyclic vector for $U$. 
\end{proof}

One very nice property of cyclic matrices is that they are similar to their companion matrices. In fact, we even know the exact conjugating matrix. This means that, given two cyclic matrices of which we know the cyclic vectors, it is very easy to check whether they are similar, and if so, also compute the conjugating matrix.

\begin{lem}\label{companionmat}
	Let $A \in \F^{n\times n}$ be cyclic with cyclic vector $v$. Furthermore let $\mathcal{B} = (v,vA,\dots, vA^{n-1})$. Then 
	$$A^{\mathcal{B}^{-1}} = M_{\chi_A} = M_{\mu_A}.$$
	In particular, $A$ is similar to $M_{\chi_A}$.
\end{lem}

\begin{citeproof}
	See \cite[p. 177]{Har70}.
\end{citeproof}

We now state a few equivalent characterisations for the cyclicity of $A$ which will be more convenient to work with.

\begin{theorem}\label{cyclicequiv}
	Let $A \in \F^{n\times n}$. Then the following are equivalent:
	\begin{enumerate}
		\item $A$ is cyclic.
		\item $A$ is similar to the companion matrix $M_{\chi_A}$.
		\item $\chi_A = \mu_A$.
		\item $\deg(\mu_A) = n$.
	\end{enumerate}
\end{theorem}

We will prove this theorem later, when we will have built up a bit more of theory. 

We see that for cyclic matrices, sharing a minimal or a characteristic polynomial is sufficient for similarity. 

\begin{kor}
	Let $A, B \in \F^{n \times n}$ be two cyclic matrices. Then the following are equivalent: 
	\begin{enumerate}
		\item $A$ and $B$ are similar. 
		\item $\mu_A = \mu_B$.
		\item $\chi_A = \chi_B$.
	\end{enumerate}
\end{kor}

\begin{proof}
	\enquote{1.$\implies$2.}: See \cref{simpol}. 
	
	\enquote{2.$\implies$3.}: Since $A$ and $B$ are cyclic we have 
	\[\chi_A = \mu_A = \mu_B = \chi_B.\]
	
	\enquote{3.$\implies$1.}: Since $A,B$ are cyclic, there exist matrices $C,D \in \GL{n}{\F}$ such that $C^{-1}AC = M_{\chi_A} = M_{\chi_B} = D^{-1}BD$. It follows that $D^{-1}C^{-1}ACD = (CD)^{-1}ACD = B$, so $A$ and $B$ are similar. 
\end{proof}

%% file: parts/02jnf.tex
\chapter{The elementary cyclic decomposition}

To define the Jordan normal form of a given matrix $A \in \GL{n}{\F}$, we will first decompose $V$ into primary spaces, that is $A$-invariant subspaces such that the restrictions of $A$ on those subspaces each have an irreducible factor with the associated multiplicity of the minimal polynomial of A as their respective minimal polynomial. Then we will see that each of these primary subspaces can be uniquely decomposed into cyclic subspaces. Finally these cyclic subspaces can be uniquely written in a form such that the companion matrix of their minimal polynomial factors are on the diagonal. We call those the Jordan blocks. 

\section{Primary decomposition}

Given a matrix $A$, computing the kernel of a polynomial of $A$ is a good way to find $A$-invariant subspaces of $V$. 

\begin{prop}
	Let $p \in \F[X]$ and $A\in \F^{n\times n}$. Then $\kernel{p(A)}$ is an $A$-invariant subspace of $V$.
\end{prop}

\begin{proof}
	Let $v \in \kernel{p(A)}$, i.e. $vp(A) = 0$. Then:
	$$(vA)p(A) = v(A\cdot p(A)) = (vp(A))A = 0\cdot A = 0.$$
	So we have $vA \in \kernel{p(A)}$.
\end{proof}

By choosing the polynomials as coprime factors of the minimal polynomial, we can get a decomposition of $V$ such that the restriction of $A$ to the subspaces has the respective factor as its minimal polynomial.

\begin{lem}\label{simplelemma}
	Let $A \in \F^{n\times n}$ and suppose there exists a factorisation of the minimal polynomial $\mu_A = p\cdot q$ into coprime factors $p, q \in \F[X]$. It then follows that 
	\begin{enumerate}
		\item $ \kernel{p(A)} = \bild{q(A)}$.  
		\item $V = \kernel{p(A)} \oplus \kernel{q(A)}$.
		\item Now suppose $p = f^i$ for some irreducible $f \in \F[X]$ and $i \in \mathbb{N}$. Furthermore let $A_p := A|_{\kernel{p(A)}}$ and $A_q := A|_{\kernel{q(A)}}$. Then $\mu_{A_p} = p$ and $\mu_{A_q} = q$.
	\end{enumerate}
\end{lem}

\begin{proof}
	Since $p$ and $q$ are coprime we know from Bézout's lemma that there exist $a,b \in \F[X]$ such that $1 = pa + qb$. 
	\begin{enumerate}
		\item \enquote{$\supseteq$}: Let $v \in V$. Then:
		\begin{align*}
			0 &= v\mu_A(A) \\& 
			   = v(p\cdot q)(A) \\&
			   = v(p(A)\cdot q(A)) \\&
			   = (vq(A))\cdot p(A).
		\end{align*}
		So $vq(A) \in \kernel{p(A)}$ for all $v \in V$ and as such $\bild{q(A)} \subseteq \kernel{p(A)}$.
		
		\enquote{$\subseteq$}: Let $v \in \kernel{p(A)}.$ Then: 
		\begin{align*}
			v &= v(pa+qb)(A) \\
			  &= \underbrace{vp(A)a(A)}_{=0} + v(qb)(A) \\
			  &= v(bq)(A)\\
			  &= (vb(A))q(A).
		\end{align*}
		So $v \in \bild{q(A)}$, for all $v \in V$ and as such $\kernel{p(A)} \subseteq \bild{q(A)}$. 
		
		\item Let $v \in V$. It follows that: 
		\begin{align*}
			v &= vI_n\\&
			   = v(pa + qb)(A) \\&
			   = v(ap + bq)(A) \\&
			   =\underbrace{(va(A))p(A)}_{\in\bild{p(A)}=\kernel{q(A)}}+\underbrace{(vb(A))q(A)}_{\in\bild{q(A)}=\kernel{p(A)}}.
		\end{align*}
	
		So $V = \kernel{p(A)} + \kernel{q(A)}$. 
		
		Now let $v \in \kernel{p(A)} \cap \kernel{q(A)}$. Then as above:
		$$v = \underbrace{vp(A)}_{=0}a(A)+\underbrace{vq(A)}_{=0}b(A) = 0.$$
		So $V = \kernel{p(A)} \oplus \kernel{q(A)}$.
		
		\item Since $p(A)|_{\kernel{p(A)}} = 0$ it follows that $\mu_{A_p} \mid p$. Similarly $\mu_{A_q} \mid q$ and thus $\mu_{A_p}$ $\mu_{A_q}$ are also coprime. 
		
		Conversely, let $v \in V$. Since we showed $V = \kernel{p(A)} \oplus \kernel{q(A)}$, we can write $v = v_p + v_q$ for $v_p \in \kernel{p(A)}$ and $v_q \in \kernel{q(A)}$. Then
		
		\begin{align*}
		v(\mu_{A_p} \cdot \mu_{A_q})(A) &= (v_p + v_q)(\mu_{A_p} \cdot \mu_{A_q})(A)
		\\&= v_p(\mu_{A_p} \cdot \mu_{A_q})(A) + v_q(\mu_{A_p} \cdot \mu_{A_q})(A)
		\\&= \underbrace{v_p\mu_{A_p}(A)}_{= 0}\mu_{A_q}(A) + \underbrace{v_q\mu_{A_q}(A)}_{= 0}\mu_{A_p}(A)
		\\& = 0.
		\end{align*}
		
		Since this is true for all vectors $v \in V$, it follows that $(\mu_{A_p}\cdot\mu_{A_q})(A) = 0$ and as such $\mu_A \mid (\mu_{A_p}\cdot\mu_{A_q})$ and furthermore $p \mid (\mu_{A_p}\cdot\mu_{A_q})$. So since $p$ and $q$ are coprime it follows that $p\nmid \mu_{A_q}$ and thus $p \mid \mu_{A_p}$ since $p$ is a power of an irreducible factor. 
		
		Thus $\mu_{A_p} = p$. Furthermore, we know that 
		\begin{align*}
			p \cdot q &=  \mu_A = \mu_{A_p \oplus A_q}\\
					  &=  \lcm(\mu_{A_p},\mu_{A_q})\\
					  &=  \mu_{A_p} \cdot \mu_{A_q}\\
					  &=  p \cdot \mu_{A_q}.
		\end{align*}
		 So $\mu_{A_q} = q$. 
	\end{enumerate}
\end{proof}

We can easily generalise this to the factorisation of the minimal polynomial into more than two factors. 

\begin{kor}\label{multipledecomp}
	Let $A \in \F^{n\times n}$ and suppose there exists a factorisation of the minimal polynomial $\mu_A = \prod_{i=1}^{\ell}p_i$ with $p_1,\dots,p_{\ell}$ coprime. Then 
	\begin{enumerate}
		\item $V = \bigoplus_{i=1}^{\ell} \kernel p_i(A)$
		\item $\kernel p_i(A) = Im(q_i(A))$ for $q_i := \prod_{j\neq i}^{\ell}p_j^{m(j)}$. 
		\item For $A_i := A|_{\kernel{p_i(A)}}$ we have $\mu_{A_i} = p_i$.
	\end{enumerate}
\end{kor}

\begin{proof}
	This follows from induction over $\ell$ by using \cref{simplelemma} since for any factorisation of a polynomial $p = \prod_{i=1}^{\ell}$ into coprime factors we can write $p = \prod{i=1}^{\ell-1}p_i \cdot p_{\ell}$ with $\prod_{i=1}^{\ell-1}p_i$ and $p_{\ell}$ being coprime.  
\end{proof}

\begin{defi}
	Let $A \in \GL{n}{\F}$. We say that $A$ is \textbf{primary} if 
	$\mu_A = p^m$ for some irreducible $p \in \F[X]$ and $m \in \mathbb{N}$.
\end{defi}

Since we know from \ref{samefactors} that the minimal and characteristic polynomial both share the same distinct irreducible factors, an equivalent definition for a primary matrix would be to say that $\chi_A = p^c$ for some irreducible $p \in \F[X]$ and $c \in \mathbb{N}$.

By factorising $\mu_A$ into distinct irreducible factors, we can use \cref{multipledecomp} to can obtain a matrix $C \in \GL{n}{\F}$ such that $A^C$ is the direct sum of primary matrices.

\begin{theorem}[Primary Decomposition]\label{primarydecomp}
	Let $A \in \F^{n \times n}$ and $\mu_A = \prod_{i=1}^{\ell}p_i^{m(i)}$ be the factorization of the minimal polynomial into monic, irreducible and mutually distinct polynomials $p_i$ and let $$V_i := \kernel(p_i^{m(i)}(A)).$$ It then follows that
	
	\begin{enumerate}
		\item $V_i$ is $A$-invariant for all $i \in \{1, \dots \ell\}$ and
		$$V = \bigoplus_{i=1}^{\ell} V_i.$$
		We call $V_1,\dots,V_{\ell}$ the \textbf{primary subspaces} of $V$ in regards to $A$. 
		
		\item $V_i = \bild(q_i(A))$ for $q_i := \prod_{j \neq i}^{\ell} p_j^{m(j)}$.
		
		\item Let $A_i := A|_{V_i}$. Then $\mu_{A_i} = p_i^{m(i)}$. 
	\end{enumerate} 
\end{theorem}

\begin{proof}
	Since any polynomial can be factorised into irreducible factors, we can write $\mu_{A} = \prod_{i=1}^{\ell} p_i^{m(i)}$ as described in the theorem for any matrix $A \in \F^{n\times n}$ by grouping their multiplicites. Then 1., 2. and 3. follow directly from \cref{multipledecomp} since the $p_i^{m(i)}$ are all coprime because of their irreducibility. 
\end{proof}

It is clear that the decomposition of $V$ into primary subspaces is unique up to reordering of direct summands. 

The naming of the primary subspaces is motivated by the following fact:

By choosing bases $\mathcal{B}_1,\dots,\mathcal{B}_{\ell}$ of $V_1,\dots,V_{\ell}$, we get a basis $\mathcal{B}$ of $V$ such that 
$$A^\mathcal{B} = \begin{pNiceArray}{cccc}
	A|_{V_1} & 0 &\dots & 0\\
	0 & A|_{V_2} & & 0\\
	\vdots & & \ddots & \vdots\\
	0 & \dots & 0 & A|_{V_{\ell}}
\end{pNiceArray}.$$ 
where, according to \cref{primarydecomp}, each $A_i$ is primary.
We see that the blocks depend only on our choice of bases for the $V_i$, i.e. they are unique up to similarity.

Once we have decomposed $V$ into primary subspaces with respect to $A$, we can work with these subspaces individually by considering the restriction of $A$ on them. 

The primary decomposition theorem has a very useful consequence. 

\begin{prop}\label{maxvec}
	
	Let $A \in \F^{n\times n}$. There exists a vector $v \in V$ such that $\mu_{A,v} = \mu_A$. We call $v$ a \textbf{maximal vector}.
\end{prop}

\begin{proof}
		We first show the special case for $A \in \F^{n\times n}$ primary, i.e. $\mu_A = p^m$ for some irreducible $p \in \F[X]$. 
		Suppose that there exists no vector $v \in V$ with $\mu_{A,v} = p^m$. Since the minimal polynomial of a vector must divide the minimal polynomial of a matrix, there exists an $r < m$ such that $vp^r(A) = 0$ for all $v \in V$. But that implies $p^r(A) = 0$ which is a contradiction to $p^m$ being the minimal polynomial of $A$.
		
		Now consider an arbitrary matrix $A \in \F^{n\times n}$ and let $\mu_A = \prod_{i=1}^{\ell}p_i^{m(i)}$ be the factorization of the minimal polynomial into irreducible, distinct polynomials $p_i$. Then, by \cref{primarydecomp} we have 
		$$V = \oplus_{i=1}^{\ell}V_i, \text{ with } V_i := \kernel(p_i^{m(i)}(A))$$
		along with a matrix $\mathcal{B}\in \GL{n}{\F}$ such that 
		$$A^{\mathcal{B}^{-1}} = \bigoplus_{i=1}^{\ell} A|_{V_i,\mathcal{B}} \text{ and } \mu_{A|_{V_i}} = p_i^{m(i)}.$$
		
		Now let $d_i := \dim(V_i)$ and $A_i := A|_{V_i,\mathcal{B}}$. We have already shown that for all $i \in \{1,\dots, \ell\}$ there exist $\tilde{v}_i \in \F^{d_i}$ such that $\mu_{A_i,\tilde{v}_i} = p_i^{m(i)}$. So there exist $v_i \in V_i$ such that $\mu_{A^{\mathcal{B}^{-1}},v_i} = p_i^{m(i)}$. 
		
		Consider $v := \sum_{i=1}^{\ell} v_i$ and let  $f := \mu_{A^{\mathcal{B}^{-1}},v}$, so $vf(A^{\mathcal{B}^{-1}})=0$. Then for all $i \in \{1,\dots, \ell\}$ we have $vf(A_i) = 0$ and thus $p_i^{m(i)} \mid f$. So, since $p_1,\dots,p_{\ell}$ are all coprime we know that $\prod_{i=1}^{\ell}p_i^{m(i)} \mid f$. Conversely, the minimal polynomial of a vector must divide the minimal polynomial of the matrix, so $f \mid \prod_{i=1}^{\ell}p_i^{m(i)}$. So all in all we have 
		$$\mu_{A,v\mathcal{B}^{-1}}\mu_{A^{\mathcal{B}^{-1}},v} = f = \prod_{i=1}^{\ell}p_i^{m(i)} = \mu_A.$$ 
\end{proof}

Recall \cref{cyclicequiv} about the equivalent characterisations of cyclic matrices. It is possible to prove these without the use of the results of the primary decomposition theorem, however the proof is much more straightforward using the existence of maximal vectors.

\begin{theorem}
	Let $A \in \F^{n\times n}$. Then the following are equivalent:
	\begin{enumerate}
		\item $A$ is cyclic.
		\item $A$ is similar to the companion matrix $M_{\chi_A}$.
		\item $\chi_A = \mu_A$.
		\item $\deg(\mu_A) = n$.
	\end{enumerate}
\end{theorem}

\begin{proof}
	\enquote{1.$\implies$2.}: 
	Follows directly from \cref{companionmat}.
	
	\enquote{2.$\implies$3}:
	Since similar matrices share the same minimal and characteristic polynomial, we have 
	$$\mu_A = \mu_{M_{\chi_A}} = \chi_A$$ by definition of the companion matrix. 
	
	\enquote{3.$\implies$4}:
	This follows directly, since $\deg(\chi_A) = n$. 
	
	\enquote{4.$\implies$1}:
	We have shown in \cref{maxvec} that there exists a maximal vector $v \in V$, i.e. $\mu_{A,v} = \mu_{A}$. So we also have $\deg(\mu_{A,v}) = \deg(\mu_{A}) = n$. Then $\dim(\matspan{v}) = \deg(\mu_{A,v}) = n = \dim(V)$. So $\matspan{v} = V$ which means that $v$ is a cyclic vector and thus $A$ a cyclic matrix. 
\end{proof}

\begin{bem}\label{CyclicEigenspaceDim}
	For primary matrices, i.e. $A \in \F^{n \times n}$ with $\mu_A = p^m$, cyclicity is then evidently equivalent to $m\cdot \deg(p) = n$.  
\end{bem}

When the condition above is met, primary matrices have a very nice property:

\begin{prop}\label{CyclicEigenspaceSubspaces}
	Let $A \in \GL{n}{\F}$ be cyclic with $\chi_A = \mu_A = p^m$ for some irreducible $p \in \F[X]$ and $m \in \mathbb{N}$. Furthermore define 
	$$V_i := \kernel p(A)^i \; \text{ for } i = 0,\dots,m,$$ along with $d := \deg(p)$. Then
	\begin{enumerate}
		\item if $v \in V$ is a cyclic vector, we have $V_i = \kernel p(A)^i = \matspan{vp(A)^{m-i}}_A$ for $i = 0,\dots,m$.
		\item $\dim(V_i) = i \cdot d$.
		\item all of the $A$-invariant subspaces of $V$ are given by $V_i$ for $i = 0,\dots,m$. 
	\end{enumerate} 
\end{prop}

\begin{proof}
	\begin{enumerate}
		\item Since $\mu_{A,v} = p^m$ and $p(A)^{m-i}\cdot p(A)^i = p(A)^m$, we see that $vp(A)^{m-i} \in \kernel(p(A)^i)$. 
		So $\matspan{vp(A)^{m-i}}_A \leq  \kernel p(A)^i$. 
		
		From \cref{randomfacts} we know that $\matspan{vp(A)^{m-i}}_A$ is cyclic and $A$-invariant. Furthermore, the minimal polynomial of the restriction of $A$ to $\matspan{vp(A)^i}_A$ is given by $p^{m-i}$, so we have $\dim(\matspan{vp(A)^i}_A) = (m-i)\cdot d$ by \cref{CyclicEigenspaceDim}. 
		
		The fundamental theorem on homomorphisms yields: $$\dim(V/\kernel(p(A)^i)) = \dim(\matspan{vp(A)^i}_A) = (m-i)\cdot d.$$
		Because $V$ is cyclic, this in turn means that
		\begin{align}
			\dim(\kernel(p(A)^i)) &= \dim(V) - \dim(V/\kernel(p(A)^i)) \\
								  &= m\cdot d - (m-i)\cdot d  =i\cdot d \label{dimsize}\\
								  &= \dim(\matspan{vp(A)^{m-i}}_A).
		\end{align}	
		
		So we have $\ker(p(A)^i) = \matspan{vp(A)^{m-i}}_A$.
		\item Follows directly from \cref{dimsize}.
		
		\item It is clear that $\{0\} = V_0 \leq \cdots \leq V_m = V$. 
		
		We first show that $V_i \neq V_{i-1}$ for all $i = 1,...,m$.
		Let $u \in V$ be a maximal vector of $A$, i.e. $\mu_{A,u} = p^m$. So $u \in V=V_m$ but $u \notin V_{m-1}$. Now let $u' := up(A)$. Then $\mu_{A,u'} = p^{m-1}$, so $u' \in V_{m-1}$ but $u' \notin V_{m-2}$. We can then repeat this process until we reach $V_m = \{0\}$. This tells us that for each $i = 0, \dots, m$ we can find an element in $V$ with minimal polynomial $p^i$. 
		
		Now let $W \leq V$ be an non-trivial $A$-invariant subspace of $V$. Then there exists $i \in \{0,\dots,m\}$ minimal such that $W \leq V_i$. 
		
		Let $w \in W\backslash V_{i-1} \subseteq V_i$. Such an element exists since $W \not\leq V_{i-1}$. Then $\mu_{A,w} = p^i$, so we have $\dim(\matspan{w}_A) = i \cdot d$. Since $V_i = \matspan{vp(A)^{m-i}}_A$ is cyclic, we also know that $\dim(V_i) = i \cdot d$. 
		
		So we get $W = V_i$. 
	\end{enumerate}
\end{proof}

\begin{bsp}
	The second statement in \cref{CyclicEigenspaceSubspaces} does not hold in general when the condition of $\mu_A = \chi_A$ is not met. For a counterexample, consider the following matrix:
	$$A = \begin{pmatrix}
		2 & 0 \\
		0 & 2 
	\end{pmatrix} \in \GL{2}{\F_3}.$$ 
	$A$ is not cyclic because $\mu_A = X-1 \neq X^2 + X + 1 = \chi_A$. 
	
	Now let $p = X-1$. The $A$-invariant subspaces $V_i$ as described in \cref{CyclicEigenspaceSubspaces} are:
	\begin{itemize}
		\item $V_0 = \kernel(p(A)^0) = \{0\}$
		\item $V_1 = \kernel(p(A)^1) = V$
	\end{itemize}
	However, it is clear that $\matspan{\begin{pmatrix}1 \\ 0\end{pmatrix}}_A$ is an $A$-invariant subspace of $V$ but neither equal to $\{0\}$ nor to $V$ since $\begin{pmatrix}1 \\ 0\end{pmatrix}A = \begin{pmatrix}2 \\ 0\end{pmatrix}$. 
\end{bsp}

\section{Cyclic decomposition}

We have seen that primary matrices have very nice properties when they are cyclic. Naturally, the question arises if and how one can decompose any primary matrix into cyclic matrices. 

\begin{lem}\label{directsumlemma}
	Let $A \in \GL{n}{\F}$ with minimal polynomial $\mu_A = p^m$ for some irreducible $p \in \F[X]$ and $m \in \mathbb{N}$. Now let $v, w \in V$ with $\mu_{A,v} = p^r$ and  $\mu_{A,w} = p^s$, such that $s\leq r$. 
	Then there exists a $w' \in V$ such that:
	$$\matspan{v}_A + \matspan{w}_A = \matspan{v}_A \oplus \matspan{w'}_A.$$
\end{lem}

\begin{proof}
	Since $\matspan{w}_A$ is cyclic, we know from \cref{CyclicEigenspaceSubspaces} that its only $A$-invariant subspaces are $W_i := \matspan{w}_Ap^i(A)$ with $\{0\} = W_s \leq W_{s-1} \leq \dots \leq W_1 \leq \matspan{w}_A$. The same is true for $\matspan{v}_A$ with subspaces $V_i := \matspan{v}_Ap^i(A)$, $i = 1,\dots,r$.
	
	If $\matspan{v}_A \cap \matspan{w}_A = \{0\}$ we may choose $w' = w$. Therefore, suppose that $S := \matspan{v}_A \cap \matspan{w}_A \neq \{0\}$. Then $S$ is a non-trivial $A$-invariant subspace of $\matspan{w}_A$, so $S \neq \{0\}$ is equivalent to the smallest non-trivial $A$-invariant subspace of $\matspan{w}_A$, namely $W_{s-1}$, lying in $S$. Following the same argument, $S \neq \{0\}$ is also equivalent to $V_{r-1}\leq S$.  
	
	Since $V_{r-1} \leq S \leq \matspan{w}_A$ we know that either $V_{r-1} \leq W_{s-1}$ or $W_{s-1} \leq V_{r-1}.$ Without restriction, let the former be the case. But since $W_{s-1}$ is already the smallest non-trivial, $A$-invariant subpace of $\matspan{w}_A$ and $V_{r-1} \neq \{0\}$ we have $V_{r-1} = W_{s-1}$, i.e. 
	$$\matspan{v}_Ap(A)^{r-1} = \matspan{w}_Ap(A)^{s-1}.$$
	This is equivalent to the existence of a $q := \alpha_{r-s}X^{r-s} + \dots + \alpha_0X^0 \in \F[X]$ such that 
	\begin{equation}\label{lem_linearkombination}
		vq(A)p(A)^{r-1} = wp(A)^{s-1}.
	\end{equation}
	We can find $\alpha_0, \dots, \alpha_{r-s}$ by solving a system of linear equations. 
	
	Now let $w' := w - vq(A)p(A)^{r-s}$. We know that the minimal polynomial of $w'$ must be $p^t$ for some $t < s$ since 
	\begin{align*}
		w'p(A)^{s-1} &= (w - vq(A)p(A)^{r-s})p(A)^{s-1}
		\\&= wp(A)^{s-1} - vp(A)^{s-1}q(A)p(A)^{r-s} 
		\\&\overset{(\ref{lem_linearkombination})}{=} wp(A)^{s-1} - vp(A)^{r-1}q(A) 
		\\&= wp(A)^{s-1} - wp(A)^{s-1}
		\\&= 0.
	\end{align*}
	Furthermore, we also have $w' = \underbrace{w}_{\in \matspan{w}_A}-\underbrace{vq(A)p(A)^{m-r}}_{\in \matspan{v}_A}$ so $w' \in \matspan{v}_A + \matspan{w}_A$, and thus $\matspan{v}_A + \matspan{w'}_A \subseteq \matspan{v}_A + \matspan{w}_A$.
	Similarly, $w = \underbrace{w'}_{\in \matspan{w'}_A} + \underbrace{vq(A)p(A)^{m-r}}_{\in \matspan{v}_A}$, so it follows that  $\matspan{v}_A + \matspan{w'}_A \supseteq \matspan{v}_A + \matspan{w}_A$.
	
	Together we have $\matspan{v}_A + \matspan{w}_A = \matspan{v}_A + \matspan{w'}_A.$
	
	We now either have $\matspan{v}_A \cap \matspan{w'}_A = \{0\}$ and are done or the intersection is not trivial, in which case the process can be repeated with the new element $w'$. Since $t < s$ we eventually get either a direct sum or $w'$ becomes $0$ after at most $s$ iterations.
\end{proof}

\begin{defi}\label{defilength}
	Let $A \in \GL{n}{\F}$ with $\mu_A = p^m$ for some irreducible $p \in \F[X]$ and $m \in \mathbb{N}$. Furthermore, let $v \in V$ with $\mu_{A,v} = p^r$ for some $r \leq m$. We call $\lambda(v):= r$ the $\textbf{length}$ of $v$ with respect to $A$. 
\end{defi}

This definition is well defined for all $v \in V$ since the minimal polynomial of a vector must divide the minimal polynomial of the matrix. 

Thanks to \cref{randomfacts} it is clear that for $d := \deg(p)$
$$\dim(\matspan{v}_A) = \lambda(v) \cdot d.$$
In particular, we know $v$ to be a cyclic vector if and only if $\lambda(v) = n/d$. 

\begin{prop}
	Let $A \in \GL{n}{\F}$ with the properties described in \cref{defilength}. Furthermore, let $w_1\dots,w_k \in V$ such that $\matspan{w_1}_A + \dots + \matspan{w_k}_A = V$. Then the following are equivalent.
	\begin{enumerate}
		\item $V = \matspan{w_1}_A \oplus \dots \oplus \matspan{w_k}_A$,
		\item $\sum_{i=1}^{k} \lambda(w_i)\cdot d = n$.
	\end{enumerate}
\end{prop}

\begin{proof}
	We know that $\matspan{w_1}_A+ \dots + \matspan{w_k}_A$ forms a direct sum if and only if $\sum_{i=1}^k \dim(\matspan{w_i}_A) = n$. Then the statement follows from the fact that $\dim(\matspan{w_i}_A) = \lambda(w_i) \cdot d$ for all $i \in \{1,\dots,k\}$.
\end{proof}

\begin{prop}\label{kernelrest}
	Let $A \in \GL{n}{\F}$ and $p \in \F[X]$. Furthemore let $W$ be an $A$-invariant subspace of $V$. Then 
	$$K_W := \kernel(p(A))\cap W \simeq \kernel(p(A|_W)).$$
\end{prop}

\begin{proof}
	Since the intersection of two vector spaces is a vector space, it is clear that $K_W$ is a subspace of both $W$ and $\kernel(p(A))$. 
	Now let $\mathcal{B}$ be a basis of $V$ and $\alpha \in \mathrm{End}(V)$ such that $A = M_{\mathcal{B}}(\alpha)$. Then \cref{restrictionop} tells us that
	\begin{align*}
		\kernel(p(A))\cap W &= \{w \in W\mid wp(A) = 0\}\\
						    &\simeq \kernel{p(A|_W)}.
	\end{align*}
\end{proof}

\begin{lem}\label{kernschranke}
	Let $A \in \GL{n}{\F}$ be a primary matrix with $\mu_A = p^m$ for some irreducible $p \in \F[X]$ and $d := \deg(p)$. Furthermore, let $w_1,\dots,w_k \in V$ such that $V = \matspan{w_1}_A + \dots + \matspan{w_k}_A$. Then 
	$$\dim(\kernel(p(A))) \leq k\cdot d$$
	with equality if and only if $V = \matspan{w_1}_A \oplus \dots \oplus \matspan{w_k}_A$.
\end{lem}

\begin{proof}
	We know that
	\begin{align*}
		\kernel(p(A)) &= \{v \in V\mid vp(A) = 0\}\\
					  &= \bigcup_{i=1}^{k}\{v_i \in \matspan{w_i}_A\mid v_ip(A) = 0\}\\
					  &= (\kernel(p(A)) \cap \matspan{w_1}_A) + \dots + (\kernel(p(A))\cap \matspan{w_k}_A).
	\end{align*}
	So from \cref{CyclicEigenspaceSubspaces} and \cref{kernelrest} we have 
	\begin{align*}
		\dim(\kernel(p(A))) &\leq \sum_{i=1}^{k}\dim(\kernel(p(A))\cap\matspan{w_i}_A)\\
							&= \sum_{i=1}^k\dim(\kernel(p(A|_{\matspan{w_1}_A})))\\
							&=k\cdot d.
	\end{align*}
	It is clear that if $V = \matspan{w_1}_A \oplus \dots \oplus \matspan{w_k}_A$, we get 
	$$\dim(\kernel(p(A))) = \sum_{i=1}^{k}\dim(\kernel(p(A))\cap\matspan{w_i}_A).$$ 
	Conversely, if the $A$-spans of $w_1,\dots,w_k$ do not form a direct sum of $V$, there must exist $i,j \in \{1,\dots k\}$ and a $w \in V\backslash\{0\}$ such that $w \in \matspan{w_i}_A \cap \matspan{w_j}_A$. Since both $\matspan{w_i}_A$ and $\matspan{w_j}_A$ are $A$-invariant, we know that $w' = wp(A)^{\lambda(w)-1} \in \matspan{w_i}_A \cap \matspan{w_j}_A$. But since $w' \in \kernel(p(A))$ we have that $w' \in \kernel(p(A)) \cap \matspan{w_i}_A \cap \matspan{w_j}_A$. So 
	$$\dim(\kernel(p(A))) < \sum_{i=1}^{k}\dim(\kernel(p(A))\cap\matspan{w_i}_A).$$
\end{proof}

We now wish to generalize \cref{directsumlemma} to cases with more than two summands. 

\begin{theorem}[Cyclic decomposition]\label{cyclicdecomposition}
	Let $A \in \GL{n}{\F}$ with $\mu_A = p^m$ for some irreducible $p \in \F[X]$ and $m \in \mathbb{N}$. Then there exist vectors $v_1, \dots, v_k \in V$ such that 
	$$V = \bigoplus_{i=1}^k \matspan{v_i}_A$$ 
\end{theorem}

\begin{proof}
	We shall prove this theorem constructively. 
	
	Let $\mathcal{W} := \{w_1,\dots,w_k\} \subseteq V$ such that $\matspan{w_1}_A + ... + \matspan{w_k}_A = V$. We know that such elements exist since any $\F$-basis of $V$ already suffices. 
	
	Now let $d := \deg(p)$. Since $\matspan{w_1}_A + ... + \matspan{w_k}_A = \matspan{w_1}_A \oplus ... \oplus \matspan{w_k}_A$ if and only if 
	$$\sum_{i=1}^k \lambda(w_i) \cdot d = n$$
	we assume that $\sum_{i=1}^k \lambda(w_i) \cdot d > n$.
	
	We then know from \cref{kernschranke} that $\dim(\kernel(p(A))) < k \cdot d$.
	
	Now consider $$B = \{w_1p(A)^{\lambda(w_1)-1}A^0,\dots,w_1p(A)^{\lambda(w_1)-1}A^{d-1}, \dots,w_kp(A)^{\lambda(w_k)-1}A^0,\dots,w_kp(A)^{\lambda(w_k)-1}A^{d-1}\}.$$
	Since $B \subset \kernel(p(A))$ and $|B| = k \cdot d$, there exists a non-trivial $\F$-linear dependence of $B$, i.e. a non-trivial $\F[X]/p\F[X]$-linear dependence of $(w_1p^{\lambda(w_1)-1}(A),...,w_kp^{\lambda(w_k)-1}(A))$. So there exist $q_1,...,q_k \in \F[X]$ with $\deg(q_i) < d$, $i = 1,...,k$, such that
	
	$$\sum_{i=1}^k w_i(q_ip^{\lambda(w_i)-1})(A) = 0.$$ 
	
	Now let $r := \min\{\lambda(w_i)|q_i \neq 0\}$ and $j \in \{1,\dots,k\}$ such that $w_j \in \{w_i \in \mathcal{W}|\lambda(w_i) = r, q_i \neq 0\}$. We then get our new set $\mathcal{W}' := \{w_1',\dots,w_k'\}$ by defining $w_i' := w_i$ for $i \neq j$ and 
	$$w_j' := \sum_{i=1}^k w_i(q_ip^{\lambda(w_i)-r})(A).$$ 
	
	We note two resulting properties of $\mathcal{W}'$:
	\begin{enumerate} 
		\item $\matspan{w_1'}_A + \dots \matspan{w_k'}_A = \matspan{w_1}_A + \dots + \matspan{w_k}_A = V$ because of the following:
		
		Per definition we have
		\begin{align}
			w_j' &= \sum_{i=1}^k w_i(q_ip^{\lambda(w_i)-r})(A) \\
				 &= w_j(q_jp^{r-r})(A) + \sum_{\substack{i=1 \\ i \neq j}}^k w_i(q_ip^{\lambda(w_i)-r})(A) \\
				 &= w_jq_j(A) + \sum_{\substack{i=1 \\ i \neq j}}^k w_i'(q_ip^{\lambda(w_i)-r})(A) \label{eqtest}.
		\end{align}
		Since $\deg(q_j) < d$ and $p$ is irreducible, we know that $q_j$ and $p$ are coprime. So by Bézout's lemma there exist polynomials $a, b \in \F[X]$ such that $$aq_j + bp = 1.$$ Then \cref{eqtest} becomes equivalent to 
		\begin{align*}
			&w_jq_j(A) = w_j' - \sum_{\substack{i=1 \\ i \neq j}}^k w_i'(q_ip^{\lambda(w_i)-r})(A) \\
			\iff &w_j(aq_j)(A) = w_j'a(A) - \sum_{\substack{i=1 \\ i \neq j}}^k w_i'(q_ip^{\lambda(w_i)-r}a)(A)\\
			\iff &w_j(1-bp)(A) = w_j'a(A) - \sum_{\substack{i=1 \\ i \neq j}}^k w_i'(q_ip^{\lambda(w_i)-r}a)(A)\\
			\iff &w_j = w_j'\underbrace{(a + bp)}_{\in \F[X]}(A) - \sum_{\substack{i=1 \\ i \neq j}}^k w_i'\underbrace{(q_ip^{\lambda(w_i)-r}a)}_{\in \F[X]}(A)
		\end{align*}
		So $w_j \in \matspan{w_1'}_A + \dots + \matspan{w_k'}_A$.
		\item $\lambda(w_j') < r = \lambda(w_j)$ because:
		\begin{align*}
			w_j'p(A)^{r-1} &= \sum_{i=1}^k w_iq_i(A)p(A)^{\lambda(w_i)-r+r-1} \\
						      &= \sum_{i=1}^k w_iq_i(A)p(A)^{\lambda(w_i)-1}\\
						      &= 0.
		\end{align*}
	\end{enumerate}

	We can repeat these steps with our new set $\mathcal{W'}$ and the process terminates after at most $\sum_{i=1}^k \lambda(w_i) - n$ such iterations. 
\end{proof}

We can see that the lengths of the vectors $v_i$ for $i = 1,\dots,k$ are decisive for the dimensions of the cyclic subspaces. To describe this more precisely we will introduce the following definition:

\begin{defi}
	Let $c \in \mathbb{N}$. We call a $k$-tuple $\lambda \coloneqq (\lambda_1,\dots,\lambda_k) \in \mathbb{N}^k$, $k \in \mathbb{N}$, such that $c = \sum_{i = 1}^{k} \lambda_i$ and $\lambda_1 \geq \dots \geq \lambda_k$ a \textbf{partition of c}. 
	Furthermore, we call $\lambda'$ with $\lambda_i' := \abs{\{j\mid\lambda_j \geq i\}}$ for $i = 1,...,k$ \textbf{the conjugate partition}.
\end{defi}
%
%
%

\begin{bem}\label{partition}
	Let $A \in \GL{n}{\F}$ such that $\chi_A = p^c$ for some irreducible $p \in \F[X]$ and $c \in \mathbb{N}$. Furthermore let $V := \matspan{v_1}_A \oplus \dots \oplus \matspan{v_k}_A$ such that $\lambda(v_1) \leq \dots \leq \lambda(v_k)$. Then the lengths of the vectors induce a partition of $c$ in the canonical way: $\lambda := (\lambda(v_1),\dots,\lambda(v_k)) \in \mathbb{N}^k$.   
\end{bem}

It is clear that the decomposition of $V$ with regards to a primary matrix $A$ is not unique. However, we can say a lot about the partition each of these decompositions induce. 

\begin{prop}\label{length}
	Let $A \in \GL{n}{\F}$, together with a decomposition of $V$ sorted by length, as described in \cref{partition}. Then, $\mu_A = p^m$, for some $m \in \mathbb{N}$ and 
	$$\lambda(v_k) = m.$$ 
\end{prop}

\begin{proof}
	Since $\chi_A = p^c$ and $p$ is irreducible it follows directly from Cayley-Hamilton, that there exists an $m \in \mathbb{N}$ such that $\mu_A = p^m$. Because the minimal polynomial of a vector divides the minimal polynomial of a matrix, it is clear that $\lambda(v_k)\leq m$.
	
	Now assume that $\lambda(v_k) < m$. Because we sorted the decomposition by length, i.e. $\lambda(v_i) \leq \lambda(v_k) < m$ for all $i = 1,\dots,k$, we have 
	$$Vp(A)^{m-1} = \bigoplus_{i=1}^k\matspan{v_ip(A)^{m-1}}_A = \{0\},$$
	so $p(A)^{m-1}=0$. This is a contradiction to $p^m$ being the minimal polynomial of $A$, hence $\lambda(v_k) =  m$.
\end{proof}

As it turns out, the partition induced on a primary space by decomposing it into cyclic subspaces is unique up to permutation. 

\begin{lem}\label{uniquepart}
	Let $A \in \GL{n}{\F}$ with $\chi_A = p^c$ and $\mu_A = p^m$ for some irreducible $p \in \F[X]$ and $c \in \mathbb{N}$. Then two decompositions of $V$ into cyclic subspaces sorted by their dimensions already induce the same partition of $c$. We call this partition the \textbf{partition of A}. 
\end{lem}

\begin{proof}	
	Consider two decompositions of $V$ into cyclic subspaces, i.e.: 
	$$V= \oplus_{i=1}^a \matspan{v_i}_A =  \oplus_{j=1}^b \matspan{w_j}_A$$
	such that $\lambda(v_1) \leq \dots \leq \lambda(v_a)$ and $\lambda(w_1) \leq \dots \leq \lambda(w_b)$. 
	
	We show by induction that
	$$L_v(k) := |\{i \in \{1,\dots,a\}\mid \lambda(v_i) = k\}| = |\{j \in \{1,\dots,b\}\mid \lambda(w_j) = k\}| =: L_w(k)$$
	for $k = 1,\dots,m$. 
	
	For $k = m$, let $I := \min\{i \in \{1,\dots,a\}\mid \lambda(v_i) = m\}$, $J := \min\{j \in \{1,\dots,b\} \mid \lambda(w_j) = m\}.$ This is well defined, since $\lambda(v_a) = \lambda(w_b) = m$ by \cref{length}. On one hand we then have:
	\begin{align*}
		Vp(A)^{m-1} &= \oplus_{i=1}^a \matspan{v_ip(A)^{m-1}}_A\\
					&= \oplus_{i=I}^a \matspan{v_ip(A)^{m-1}}_A,
	\end{align*}
	and on the other hand:
	\begin{align*}
		Vp(A)^{m-1} &= \oplus_{j=1}^b \matspan{w_jp(A)^{m-1}}_A\\
				    &= \oplus_{j=J}^b \matspan{w_jp(A)^{m-1}}_A.
	\end{align*}
	Now let $d := \deg(p)$. By comparing dimensions get:
	\begin{align*}
		\dim(Vp(A)^{m-1}) &= \sum_{i=I}^{a} \lambda(v_ip(A)^{m-1})\cdot d\\
						  &= \sum_{i=I}^a 1 \cdot d\\
						  &= (a-I)\cdot d
	\end{align*}
	along with
	\begin{align*}
		\dim(Vp(A)^{m-1}) &= \sum_{j=J}^b \lambda(w_jp(A)^{m-1})\cdot d\\
						  &= \sum_{j=J}^b \cdot d\\
						  &= (b-J)\cdot d.
	\end{align*}
	So $L_v(m) = (a-I) = (b-J) = L_w(m)$.
	
	Now suppose there exists a $k \in \{1,\dots,m-1\}$ such that $L_v(k) = L_w(k)$ for all $k' > k$.
	
	Let $I := \min\{i \in \{1,\dots,a\}\mid \lambda(v_i) \geq k\}$ and $J := \min\{j \in \{1,\dots,b\}\mid \lambda(w_j) \geq k\}$. Then 
	\begin{align*}
		Vp(A)^{k-1} &= \oplus_{i=1}\matspan{v_ip(A)^{k-1}}_A\\
					&= \oplus_{I=1}\matspan{v_ip(A)^{k-1}}_A\\
	\end{align*}
	and 
	\begin{align*}
		Vp(A)^{k-1} &= \oplus_{j=1}\matspan{w_jp(A)^{k-1}}_A\\
					&= \oplus_{J=1}\matspan{w_jp(A)^{k-1}}_A.
	\end{align*}
	So again, by comparing dimensions we get:
	\begin{align*}
		\dim(Vp(A)^{k-1}) = \sum_{i=I}^a\lambda(v_ip(A)^{k-1})_A \cdot d = \sum_{j=J}^b\lambda(w_jp(A)^{k-1})_A \cdot d.
	\end{align*}
	If we define $I' := \max\{i\in\{1,\dots,a\}\mid \lambda(v_i)\leq k\}$ and $J' := \max\{j \in \{1,\dots,b\}\mid \lambda(w_j)\leq k\}$ we can write the sums above as 
	\begin{align*}
		\sum_{i=I}^a\lambda(v_ip(A)^{k-1})\cdot d &= 
		\sum_{i=I}^{I'}\lambda(v_ip(A)^{k-1})\cdot d + \sum_{I'}^{a}\lambda(v_ip(A)^{k-1})\cdot d\\
		&= \sum_{i=I}^{I'}\lambda(v_ip(A)^{k-1}) + \sum_{\ell=k+1}^m L_v(\ell) \cdot \ell \cdot d 
	\end{align*}
	and similarly $$\sum_{j=J}^b\lambda(w_jp(A)^{k-1})\cdot d = \sum_{j=J}^{J'}\lambda(w_jp(A)^{k-1}) + \sum_{\ell=k+1}^m L_w(\ell) \cdot \ell \cdot d.$$
	Since $k+1 > k$, we have $\sum_{\ell=k+1}^m L_v(\ell) \cdot \ell \cdot d = \sum_{\ell=k+1}^m L_w(\ell) \cdot \ell \cdot d$ by our induction hypothesis, so the only vectors left to compare are those with length $k$:
	\begin{align*}
		& \sum_{i=I}^{I'}\lambda(v_ip(A)^{k-1})\cdot d = \sum_{j=J}^{J'}\lambda(w_jp(A)^{k-1})\cdot d\\
		\Leftrightarrow\text{ }& \sum_{i=I}^{I'}1 \cdot d =  \sum_{j=J}^{J'} 1 \cdot d\\
		\Leftrightarrow\text{ }& (I'-I)\cdot d = (J'-J)\cdot d.
	\end{align*}
	So we have $L_v(k) = (I'-I) = (J'-J) = L_w(k)$. 
	
	Since $(L_v(1),\dots,L_v(m))$ is precisely the conjugate partition to $(\lambda(v_1), \dots,\lambda(v_a))$ it then follows, that $v_1,\dots,v_a$ and $w_1,\dots,w_b$ must induce the same partition on $c$.
\end{proof}

\begin{bsp}
	Note that if $A$ is not primary, the dimensions of the subspaces in a decomposition of $V$ into cyclic subspaces need not be unique. For a counterexample, consider
	$$A := \begin{pmatrix}
			1 & 0 \\ 0 & 2
		  \end{pmatrix} \in \GL{2}{\F_3}.
	$$
	with $\mu_A = (X+1)(X-1)$. 
	Since $\deg(\mu_A) = 2$, $A$ is cyclic. But $A = (1) \oplus (2)$, and $(1),(2) \in \GL{1}{\F_3}$ are also both cyclic. 
\end{bsp}

\begin{bem}
	We have seen so far that for all matrices $A \in \F^{n \times n}$,  V can be uniquely decomposed into primary subspaces, i.e. $A$ is similar to a direct sum of primary matrices. By considering the vector spaces these primary matrices act on, we can again decompose each of these into cyclic subspaces. Thus, there exists a decomposition of $V = \bigoplus W_i$ into primary and cyclic subspaces with respect to $A$. The exact subspaces may not be unique, but we have shown that their dimension is. Then, since $A|_{W_i}$ is cyclic, the minimal polynomial of the matrix restrictions is also unique. We call these minimal polynomials the \textbf{elementary divisors} of $A$. 
\end{bem}

\begin{theorem}\label{elementarydiv}
	Two matrices $A,A' \in \GL{n}{\F}$ are similar if and only if they share the same elementary divisors. 
\end{theorem}

\begin{proof}
	\enquote{$\implies$}: Suppose that $A, A'$ are similar, i.e. there exists a matrix $\mathcal{B} \in \GL{n}{\F}$ such that $A^{\mathcal{B}^{-1}} = A'$. Now let $V = \oplus \matspan{w_i}_A$ be a decomposition of $V$ into primary, cyclic subspaces with respect to $A$. Then for $w_i' = w_i\mathcal{B}$, $\oplus\matspan{w_i'}_{A'}$ is a decomposition of $V$ into primary, cyclic subspaces with respect to $A'$. Since $\mu_{A,w_i} = \mu_{A^{\mathcal{B}},\mathcal{B}w_i} = \mu_{A',w_i'}$, the two matrices share the same elementary divisors.  
	
	\enquote{$\impliedby$}: Now suppose that $A$ and $A'$ share the same elementary divisors $f_1,\dots,f_k$. So by definition there exist vectors $v_1,\dots,v_k$ and $v_1',\dots,v_k'$ such that 
	\begin{enumerate}
		\item $V = \bigoplus_{i=1}^{k}\matspan {v_i}_A = \bigoplus_{i=1}^{k}\matspan{v_i'}_A$
		\item $\text{and }  \mu_{A,v_i} = \mu_{A',v_i'} = f_i, \text{ for } i = 1,\dots, k.$
	\end{enumerate}
	So there exist bases $\mathcal{B},\mathcal{B}'$ of $V$ such that 
	$$A^{\mathcal{B}^{-1}} = A^{\mathcal{B'}^{-1}} = \bigoplus_{i=1}^k M_{f_i}.$$
	Thus $A^{\mathcal{B}^{-1}\mathcal{B}'^{-1}} = A'$. 
\end{proof}

\begin{bem}
	Given a matrix $A \in \F^{n\times n}$ with elementary divisors $f_1,\dots,f_k \in \F[X]$, the matrix 
	$$\bigoplus_{i=1}^k M_{f_i}$$ 
	described in the proof of \cref{elementarydiv} is called the \textbf{primary canonical form} of $A$. 
\end{bem}

\chapter{The Jordan normal form}

\begin{lem}[Jordan Block on cyclic primary spaces]\label{Jordan Block}
	Let $A \in \GL{n}{\F}$ be both cyclic and primary, i.e. $\chi_A = \mu_A = p^m$ for some irreducible $p \in \F[X]$, $m \in \mathbb{N}$ with $d := \deg(p)$. Furthermore, let $v$ be a cyclic vector of $V$. Then $\mathcal{B} := (b_1, \dots, b_n) \in V$ with $b_1 := v$ and 
	$$b_{di+r} := vp(A)^iA^{r-1},\; 1 \leq r \leq d, 0\leq i < m$$ 
	is a basis of $V$, so
	$$\mathcal{B} = (vA^0, \dots, vA^{d-1}, vA^0p(A),\dots, vA^{d-1}p(A),\dots,vA^0p(A)^{m-1}, \dots vA^{d-1}p(A)^{m-1}).$$
	Furthermore
	$$J(p^m) := A^\mathcal{B} =
		\begin{pNiceArray}{cccccc}
		M_p & N_d & 0 & \dots & 0 & 0\\
		0 & M_p & N_d & \dots & 0 & 0\\
		0 & 0 & M_p & \dots & 0 & 0\\
		0 & 0 & 0 & \dots & 0 & 0\\
		\vdots & \vdots & \vdots & \vdots & \vdots & \vdots\\
		0 & 0 & 0 & \dots & 0 & M_p
		\end{pNiceArray} \in \F^{md\times md}
	$$ 
	with $0$ being the zero matrix in $\F^{d \times d}$, $M_p$ being the companion matrix of $p$ and $N_d \in \F^{d\times d}$ such that $(N_d)_{i,j} := \begin{cases}
		1 \text{, if } (i,j) = (1,d) \\
		0 \text{, else}
	\end{cases}.$

	We call $J(p^m)$ the \textbf{generalised Jordan block} of $p^m$. 
\end{lem}

\begin{proof}
	Let $V_i = Vp(A)^i$ for $1 \leq i < m$.  
	For a fixed $0 \leq i < m$, the set of $\{b_{r+di}+V_{i+1}: 0 \leq r \leq d-1 \}$ gives us a basis $\mathcal{B}$ of $V_i/V_{i+1}$, since $\dim(V_i/V_{i+1}) = d$. So since $V_i/V_{i+1}$ is $A$-cyclic we have:
	$$A^{\mathcal{B}^{-1}} = 
	\begin{pNiceArray}{cccc}
		M_p & C_1 & \dots & 0\\
		\Block{2-1}{\vdots} & \ddots & \ddots & \vdots\\
		 & &\ddots& C_m\\
		0 & \Block{1-2}{\dots} & & M_p 
	\end{pNiceArray} \in \F^{n\times n}.$$ 
	for some $C_i \in \F^{m \times d}$. 
	Since $b_{r+1+di} = b_{r+di}A$ for all $r \neq d$ we know that the first $d-1$ columns of $C$ are zero. Furthermore, with $p= X^d + a_{d_1}X^{d-1} + \dots + a_0$ we have that
	$$b_dA = b_1A^d = b_1p(A) - a_0b_1 - \dots a_{d-1}b_d$$
	so the $d$-th column of $C$ is given by $(1, 0, \dots, 0)^T \in \F^m$. 
\end{proof}

\begin{bem}
	Continuing with the setting of \cref{Jordan Block}, if $\F$ is an algebraically closed field such as $\mathbb{C}$, we have that $p(X) = (X-a)$ for some $a \in \F$.
	So 
	$$J(p^m) = \begin{pmatrix} 
		a & 1 &  &  & 0 \\  
		  & a & 1 &  &  \\ 
		 &\ddots& \ddots{} & \\
		  & & \ddots & a & 1 \\ 
		  0 &  & &  & a 
	  \end{pmatrix} \in \F^{n \times n},$$
  	i.e. the generalised Jordan block then coincides with the Jordan block in the classic sense.
\end{bem}

We are now ready to put all of the pieces of the puzzle together.

\begin{theorem}[The Jordan normal form]
	Let $A \in \F^{n\times n}$ with elementary divisors $f_i,\dots,f_k \in \F[X]$. Then there exists a basis $B$ of $V$ such that 
	$\mathrm{JNF}(A) := A^{B^{-1}} = \oplus_{i=1}^{k}J(f_i).$
		
	We call $\JNF(A)$ the \textbf{generalised Jordan normal form of $A$}. 
\end{theorem}

\begin{proof}
	We have shown that $V$ can be decomposed into primary cyclic subspaces $V= V_1\oplus\dots\oplus V_k$, so there exists a basis $\mathcal{B}$ of $V$ such that
	$$A^{\mathcal{B}^{-1}} = \begin{pNiceArray}{ccc}
		A_1 &  & 0\\
		& \ddots & \\
		0 &  & A_{k}
	\end{pNiceArray} \in \F^{n \times n}$$
	with $\mu_{A_i} = f_i$. Then by \cref{Jordan Block}, there exist $\mathcal{B}_1,\dots,\mathcal{B}_k$ with $\mathcal{B}_i \in \GL{\deg(f_i)}{\F}$ such that $A_i^{\mathcal{B}^{-1}_i} = J(f_i)$. So by defining the matrix 
	$$\mathcal{C} = \begin{pNiceArray}{ccc}
		\mathcal{B}^{-1}_1 &  & 0\\
		& \ddots & \\
		0 &  & \mathcal{B}^{-1}_k
	\end{pNiceArray} \in \GL{n}{\F}$$
	we see that $A^{\mathcal{B}^{-1}\mathcal{C}}$ has the desired form. 
\end{proof}

\begin{kor}\label{JNFequiv}
	Two matrices $A,B \in \F^{n\times n}$ are similar if and only if $\JNF(A) = \JNF(B)$. 
\end{kor}

\begin{proof}
	This follows directly from \cref{elementarydiv} and the fact that the Jordan normal form of a matrix is uniquely determined by its elementary divisors.
\end{proof}

%% file: parts/04alg.tex
\chapter{Computing the Jordan normal form}

We now present algorithms to compute the Jordan normal form. For practical reasons, we only consider finite fields, so for the rest of this chapter we let $q$ be the power of a prime and $\F = \F_q$ the field with $q$ elements. 

\section{Auxiliary functions}

In order to implement the algorithms for computing the Jordan normal form efficiently, it is useful to first introduce a set of auxiliary functions that perform recurring operations. They serve as essential building blocks for the main algorithmic framework.

One of the most important functions, despite its simplicity, is the spinning algorithm as described in \cite{Par84}. Given a matrix $A \in \F^{n\times n}$ and a vector $v \in \F^n$, it computes the sequence of images $\{v,vA,\dots,vA^{n-1}\}$. If $A$ is cyclic, this yields a basis of $\F^n$.

Rather than calculating each matrix power $A^i$ explicitly, we can iteratively obtain $vA^{i}$ from $vA^{i-1}$ through a single matrix-vector multiplication instead, thereby avoiding costly matrix-matrix multiplications. This optimization reduces the complexity to $\On(n^3)$ field operations.

\begin{algorithm}[H]
	\caption{Spinning algorithm}
	\KwData{$A \in \F^{n \times n},\ v \in \mathbb{F}^{1 \times n}$}
	\KwResult{A matrix with $v, vA, \dots, vA^{n-1}$ as its rows}
	\SetKwFunction{Spinning}{Spinning}
	\SetKwProg{Fn}{Function}{:}{}
	\Fn{\Spinning{$A$, $v$}}{
		spun $\gets [v]$\;
		\For{$i \gets 1$ \KwTo $n-1$}{
			$v \gets v A$\;
			append $v$ to the end of spun\;
		}
		\Return{spun}\;
	}
\end{algorithm}

Alternatively, if one knows the dimension $d$ of $\matspan{v}_A$ for a vector $v \in V$, it may be useful to only compute $v,vA,\dots$ until the $(d-1)$-th power. Then $\{v,vA,\dots,vA^{d-1}\}$ is a basis of $\matspan{v}_A$.

\begin{algorithm}[H]
	\caption{Spinning algorithm until $d$}
	\KwData{$A \in \F^{n \times n},\ v \in \mathbb{F}^{1 \times n}$}
	\KwResult{A matrix with $v, vA, \dots, vA^{d-1}$ as its rows}
	\SetKwFunction{SpinUntil}{SpinUntil}
	\SetKwProg{Fn}{Function}{:}{}
	\Fn{\SpinUntil{$A$, $v$, $d$}}{
		spun $\gets [v]$\;
		\For{$i \gets 1$ \KwTo $d-1$}{
			$v \gets v A$\;
			append $v$ to the end of spun\;
		}
		\Return{spun}\;
	}
\end{algorithm}

The algorithm then takes $\On(n^2\cdot d)$ field operations. 

Given a cyclic matrix $A \in \F^{n \times n}$, we can test a random vector $v \in \F^n$ for cyclicity by using the spinning algorithm on it and checking whether the matrix one obtains has full rank, i.e. whether its rows form a basis of $V$. Since this also directly yields a basis of $V$ in terms of $v$, it is useful to return it alongside the vector. 

Clearly, the efficiency of such an algorithm depends on the probability of a randomly chosen vector being cyclic. In that regard, Neumann and Praeger showed a very useful estimation of that probability.

\begin{lem}\label{cyclicprob}
	Let $A \in \F^{n\times n}$ and let $p_1,\dots,p_s$ be the distinct irreducible polynomials which divide the minimal polynomial $\mu_A$ and let $d_i := \deg(p_i)$. Then the proportion of vectors in $V$ which are cyclic vectors for $A$ is 
	$$\mathcal{P}_c = \prod_{i=1}^s (1-\frac{1}{q^{d_i}})$$
\end{lem}

\begin{proof}
	See \cite[Theorem 2.2]{NeChe95}
\end{proof}

Let $A$, $p_1,\dots,p_s$ and $d_1,\dots,d_s$ as in \cref{cyclicprob}.
Given an $\varepsilon \in [0,1)$ we can make an estimation for the number $N$ of tries within which we can guarantee finding a cyclic vector for $A$ with probability $\varepsilon$. We have

\begin{align*}
	\mathcal{P}_c &= \prod_{i=1}^s 1-q^{-d_i}\\
				  &\geq \prod_{i=1}^s 1-q^{-1}\\
				  &= (1-q^{-1})^s\\
				  &\geq 1-q^{-1}
\end{align*}

So if $\mathcal{P}_f = (1-\mathcal{P}_c)^N$ is the probability of failure for finding a cyclic vector within $N$ tries, we have $\mathcal{P}_f < \varepsilon$ if and only if $N > \log((1-\varepsilon)^{-1}) \cdot \log(q)^{-1}$.
In particular we see that for growing field size, the amount of tries we need to find a cyclic vector becomes very small. Even for $q = 9$ and $\varepsilon = 0.99$ we already have $N > \frac{\log(100)}{\log(9)} \approx 2.095$.

\begin{algorithm}[H]\caption{Find cyclic vector}\label{findcyclicvec}
	\KwData{$A \in \mathbb{F}^{n \times n}$ cyclic, $\varepsilon$ probability of finding a cyclic vector}
	\KwResult{A matrix with $v,vA,\dots,vA^{n-1}$ as its rows for a cyclic vector $v$}
	\SetKwFunction{FindCyclicVector}{FindCyclicVector}
	\SetKwProg{Fn}{Function}{:}{end}
	\Fn{\FindCyclicVector{$A$}}{
		\For{$i \gets 1$ \KwTo $\ceil{\frac{\log((1-\varepsilon)^{-1})}{\log(q)}}$}{
			$v \gets$ Random($\mathbb{F}^n\backslash\{0\}$)\;
			$generators \gets$ \Spinning{$A, v$}\;
			\If{$\mathrm{Rank}(generators) = n$}{
				\Return{generators}\;
			}
		}
	}
\end{algorithm}

Given an $\varepsilon \in [0,1)$, \cref{findcyclicvec} will make at most $\ceil{\frac{\log((1-\varepsilon)^{-1})}{\log(q)}}$ attempts to find a cyclic vector. Since each try we perform once the spinning algorithm and compute once the rank of the resulting $\F^{n\times n}$ matrix, which costs $\On(n^3)$ and $\On(n^2)$ respectively, each try takes $\On(n^3)$ field operations. Thus, this algorithm costs $\On(\ceil{\frac{\log((1-\varepsilon)^{-1})}{\log(q)}}\cdot n^3)$ field operations. From here on, we choose $\varepsilon = 0.99$, so $\frac{\log(100)}{\log(q)} \leq 6$ for all $q \geq 2$ and thus \cref{findcyclicvec} costs $\On(n^3)$ field operations. 

Using a similar strategy as in the spinning algorithm, we can also efficiently evaluate matrix in polynomials if the result is also multiplied with a vector.

\begin{algorithm}[H]
	\caption{Evaluate matrix in polynomial multiplied with vector}
	\label{polyeval}
	\KwData{$A \in \F^{n \times n}$, $p(X) = a_0 + \dots + a_{d-1}X^{d-1} \in \F[X]$, $v \in \F^n$}
	\KwResult{$vp(A) \in \F^n$}
	\SetKwFunction{EvaluatePolynomialWithVec}{EvaluatePolynomialWithVec}
	\SetKwProg{Fn}{Function}{:}{end}
	\Fn{\EvaluatePolynomialWithVec{$A,p,v$}}{
	$v_1 \gets a_{d-1} \cdot v$\;
	\For{$i \gets d-2$ \KwTo $0$}{
		$v_1 \gets v_1 \cdot A$\;
		\If{$a_i \neq 0$}{
			$v_1 \gets v_1 + a_i \cdot v$\;
		}
	}
	\Return{$v_1$}\;
	}
\end{algorithm}

While the straightforward computation of $vp(A)$ would require $\On(n^3\cdot\frac{d(d+1)}{2})$ field operations. In contrast, this approach reduces the cost to $\On(n^2\cdot d)$ field operations.

\begin{algorithm}[H]
	\caption{Evaluate matrix in polynomial with spun vector}\label{nicepolyeval}
	\KwData{$A \in \F^{n \times n}$, $p(X) = a_0 + \dots + a_{d-1}X^{d-1} \in \F[X]$ with $d \leq n-1$, $\{v,vA,\dots,vA^{n-1}\}$}
	\KwResult{$vp(A) \in \F^n$}
	\SetKwFunction{EvaluatePolynomialWithSpun}{EvaluatePolynomialWithSpun}
	\SetKwProg{Fn}{Function}{:}{end}
	\Fn{\EvaluatePolynomialWithSpun{$A,p,vSpan$}}{
		$v_1 \gets a_{d-1} \cdot v$\;
		\For{$i \gets d-2$ \KwTo $0$}{
			$v_1 \gets vSpan[i]$\;
			\If{$a_i \neq 0$}{
				$v_1 \gets v_1 + a_i \cdot v$\;
			}
		}
		\Return{$v_1$}\;
	}
\end{algorithm}

Moreover, if the spinning algorithm has already been applied in advance, i.e. the set ${v,vA,\dots,vA^{n-1}}$ is available, and $\deg(p) = d \leq n-1$, the algorithm can be adapted to directly use these precomputed vectors.
This simple modification lowers the computational effort even further to $\On(n\cdot d)$ field operations.

\section{Primary Decomposition}

The first step in computing the Jordan normal form of a matrix $A$ is computing the primary decomposition of the vector space with respect to $A$. An efficient method to do this has been proposed by Allan Steel  \cite{St97}, avoiding the direct evaluation of matrices in polynomials.

\begin{algorithm}[H]
	\caption{Primary Decomposition (Modified version of Steel's algorithm)}
	\KwIn{A matrix $A \in \F^{n\times n}$}
	\KwOut{Matrix $B$ such that $BAB^{-1}$ is in primary decomposition form}
	\SetKwFunction{Ech}{EcheloniseMat}
	\SetKwFunction{Factors}{Factors}
	\SetKwFunction{Collected}{Collected}
	\SetKwFunction{PrimaryDecomposition}{PrimaryDecomposition}
	\SetKwProg{Fn}{Function}{:}{end}
	\Fn{\PrimaryDecomposition{$A$}}
	{
		$v \gets$ random vector in $\F^n$\;
		$rank \gets 0$\;
		$\mathcal{L} \gets [\,]$; $\mathcal{G} \gets [\,]$\;
		
		\While{$rank < n$}{
			$m(x) \gets$ $\mu_{A,v}$\;
			$p(x) \gets 1$ \tcp*{Product of known primary components}
			\ForEach{$i = 1$ \KwTo $\text{Size}(\mathcal{L})$}{
				$g_i(x) \gets \mathcal{G}[i]$\;
				\If{$\gcd(m, g_i) \neq 1$}{
					$e \gets$ maximal power such that $g_i^e\mid m$\;
					$f(x) \gets m(x)/g_i^e(x)$\;
					$w \gets \EvaluatePolynomialWithVec(A, f(x), v)$\;
					$p(x) \gets p(x) \cdot g_i^e(x)$\;
					$W \gets \Spinning(w, A)$\;
					$\mathcal{L}[i] \gets \Ech(\text{Concat}(W, \mathcal{L}[i]))$\;
				}
			}
			$v \gets \textbf{\EvaluatePolynomialWithVec}(A, p(x),v)$\;
			$m \gets m / p(x)$\;
			\If{$m \neq 1$}{
				$\text{factors} \gets \Collected(\Factors(m))$ \tcp*{List of $(q_i, e_i)$}
				\ForEach{$(q_i, e_i)$ in factors}{
					$f_i(x) \gets m / q_i^{e_i}$\;
					$w \gets \EvaluatePolynomialWithVec(A, f_i, v)$\;
					$W \gets \Spinning(w, A)$\;
					
					Append $W$ to $\mathcal{L}$;
					 Append $q_i$ to $\mathcal{G}$\;
				}
			}
			$C \gets$ matrix with all matrices in $\mathcal{L}$ as rows\;
			$rank \gets \text{number of rows of } C$\;
			\If{$rank < n$}{
				$v \gets$ Random vector in $\F^n\backslash\bild(C)$\;
			}
		}
		
		\Return{$C$}\;
	}
\end{algorithm}

The central idea of the algorithm is to construct vectors that have primary minimal polynomials with respect to $A$. If $\mu_A = \prod_{i=1}^{\ell} p_i^{m(i)}$ then for each distinct factor $p_i$ a sufficient number of those vectors is generated so their $A$-spans form the corresponding primary subspace. In practice, this is done by recording the $A$-spans of the aforementioned vectors in sets $L_1,\dots,L_{\ell}$, according to the minimal polynomial of the vectors, and comparing the span of the new vectors with the already recorded ones.

This approach avoids the costly direct evaluation of matrices in polynomials, which is very expensive. Instead, whenever a matrix is evaluated in a polynomial, it is also multiplied by a vector, allowing us to use \cref{polyeval}.

The original algorithm proposed by Steel additionally records a generating vector for each factor $p_i$ when the subspace $\kernel(p_i(A))$ is $A$-cyclic. However, the exact details on this have been omitted here for the sake of simplicity. Importantly, this modification does not affect the overall computational complexity, which remains in the order of $\On(n^4)$ field operations, as established in \cite{St97}.

Computing the primary decomposition is the most expensive step in computing the Jordan normal form. We will see however, that we can greatly increase the algorithm's efficiency if $A$ happens to be cyclic.

\section{Cyclic decomposition}

Once we are working with the restriction on a primary space, so we have a matrix $A \in \F^{n\times n}$ such that $\mu_A = p^m$ for some irreducible $p \in \F[X]$, we can apply the procedure outlined in the proof of \cref{cyclicdecomposition} to compute a cyclic decomposition of $V$. To facilitate this, we first introduce several auxiliary functions that will be essential in the subsequent computations.

Recall that for a primary matrix $A$, the minimal polynomial of any $v \in V$ takes the form $p^r$ for some $\lambda(v) := r \leq m$. $\lambda(v)$ is called the $A$-length of $v$.
Consequently, determining the minimal polynomial of a vector $v \in V$ is reduced to checking whether $vp(A)^i = 0$ for increasing powers $i$.

The following function computes the $A$-length $\lambda(v)$ of a vector $v$. 

\begin{algorithm}[H]
	\caption{$A$-length of a vector}
	\KwData{$p(A)$ for a matrix $A \in \F^{n\times n}$ with $\mu_{A} = p^m$ for $p \in  \F[X]$ irreducible, $v \in V$}
	\KwResult{$\lambda(v),$}
	\SetKwFunction{VectorLength}{VectorLength}
	\SetKwProg{Fn}{Function}{:}{end}
	\Fn{\VectorLength{$p(A),m,v$}}{
		\If{$v = 0$}{
			\Return $[0,v]$\;
		}
		$lastv \gets v$\;
		\For{$i \gets 1$ \KwTo $m$}{
			$v \gets vp(A)$\;
			\If{$v = 0$}{
				\Return $[i,lastv]$\;
			}
			$lastv \gets v$\;
		}
	}
\end{algorithm}

Note that computing the length $r$ of a vector $v$ also gives us $vp(A)^{r-1}$, which will be useful later, so we return both values. The computation of the $A$-length of a vector requires $\On(n^2)$ field operations. 

Another essential function is to determine an $\F[X]$-linear dependence of a set of vectors $\{v_1,\dots,v_k\}$, in the sense of $\cdot_A$ as described in \cref{famodule} for a matrix $A \in \F^{n\times n}$. We only consider the special case where each vector $v_i$ has length one. In this case, the problem reduces to finding an $\F$-linear dependence of $\{v_1A^0,\dots,v_1A^{d-1},\dots,v_kA^0,\dots v_kA^{d-1}\}$, where $d := \deg(p)$.

\begin{algorithm}[H]
	\caption{$\F[X]/p\F[X]$-linear dependence of vectors}\label{dependence}
	\KwData{$A \in \GL{n}{\F}$ with $\mu_A = p^m$ for $p \in \F[X]$ irreducible, $d = \deg(p)$, $vecs = (v_1\dots,v_k)$ with $v_i \in V$ and $\lambda(v_i) = 1$}
	\KwResult{$q_1\dots,q_k \in \F[X]$ such that $\sum_{i=1}^kv_iq_i(A)=0$, with $\deg(q)<\deg(p)$}
	\SetKwFunction{FXLinearDependence}{FXLinearDependence}
	\SetKwProg{Fn}{Function}{:}{end}
	\Fn{\FXLinearDependence{$A,d,vecs$}}{
		$toSolve \gets []$\;
		\For{$i\gets1$ \KwTo $\mathrm{Size}(vecs)$}{
			$vspan \gets \SpinUntil{A,vecs[i],d}$\;
			Append $vspan$ to the end of $toSolve$\;
		}
		$relationBasis \gets$ basis of $\kernel(toSolve)$\;
		\If{$relationBasis$ is empty}{
			Print(No linear dependence could be found.)\;
			\Return 0\;
		}
		$relation \gets relationBasis[1]$\;
		$qis \gets []$\;
		\For{$i\gets1$ \KwTo $\mathrm{Size}(vecs)$}{
			$qi \gets$ polynomial with $[relation[i\dots (i\cdot d)-1]]$ as coefficients\;
			Append $qi$ to end of $qis$\;
		}
		\Return $qis$\;
	}
\end{algorithm}

Finding such a linear dependence requires $\On(n^2)$ field operations.

With these auxiliary procedures at hand, we are now prepared to tackle the main task of this section: computing the decomposition of a primary space into a direct sum of cyclic subspaces. As mentioned before, we proceed following the procedure outlined in the proof of  \cref{cyclicdecomposition}. We begin by initializing a set of vectors that span the entire space and computing the $A$-length of each of these vectors. Then, using \cref{dependence} we compute their $\F[X]/p\F[X]$-linear dependence to compute the new set of generating vectors. We repeat this step iteratively until the $A$-spans of all of the vectors form a direct sum decomposition of the space. 

The computational cost of decomposing a primary matrix $A$ into a direct sum of cyclic matrices is dominated by three main tasks:

\begin{enumerate}
	\item Selecting generators $w_1,\dots,w_k$ whose combined $A$-span equals $V$,
	\item determining their $A$-lengths,
	\item and finding an $\F[X]$ linear dependency to reduce the total lengths of the generating vectors. 
\end{enumerate}

Finding the generating vectors involves repeated spinning of vectors and echelonising their span, performed at most $n$ times, for a total cost of $\On(n^4)$ field operations.
Once the generators are found, computing their A-lengths incurs an additional $\On(k\cdot n^2)$ operations.

\begin{algorithm}[H]
	\caption{Cyclic decomposition of primary space}
	\KwData{$A \in \F^{n\times n}$ with $\mu_{A} = p^m$ for $p \in \F[X]$ irreducible}
	\KwResult{Matrix $B$ such that $BAB^{-1} = A_1\oplus\dots \oplus A_k$ with $A_i$ cyclic}
	\SetKwFunction{CyclicDecomposition}{CyclicDecomposition}
	\SetKwProg{Fn}{Function}{:}{end}
	\Fn{\CyclicDecomposition{$A,p,m$}}{
		$d \gets \deg(p)$\;
		\lIf{$m\cdot d = n$}{\Return $I_n$}
		$w \gets$ Random vector in $\F^n\backslash\{0\}$\;
		$ws \gets [$w$]$ \tcp*{List of generating vectors}
		$gens \gets$ \Spinning($A,w$)\;
		Echelonise $gens$\;
		\While{\texttt{Size}$(gens)<n$}{
			$w \gets$ Random vector in $\F^n\backslash\matspan{gens}$\;
			Append \Spinning($A,w$) to end of $gens$\;
			Echelonise $gens$\;
			Append $w$ to end of $ws$\;
		}
		$k \gets \abs{ws}$\;
		$M \gets [\;]$\;
		\For{$i \gets 1$ \KwTo $k$}{
			$[w_i,\lambda(w_i),wp(A)^{\lambda(w_i)-1}] \gets \VectorLength(A,p,ws[i])$\; 
			Append $[w_i,\lambda(w_i),w_ip(A)^{\lambda(w_i)-1}]$ to end of $M$\;
		}
		\While{$\sum_{i=1}^k \lambda(w_i)\neq \frac{n}{d}$}{
			$dependentVecs \gets [w_kp(A)^{\lambda(w_k)-1},\dots, w_kp(A)^{\lambda(w_k)-1}]$\;
			$q_1,\dots, q_k \gets$ \FXLinearDependence{$A,d,dependentVecs$}\;
			$\lambda(w_j) \gets \min\{\lambda(w_1),\dots,\lambda(w_k)\}$\;
			$w' \gets \sum_{i=1}^k w_i\cdot q_i(A)\cdot p(A)^{\lambda(w_i) - \lambda(w_j)}$\;
			\uIf{$w' = 0$}{
				Remove $M[j]$\;
				$k \gets \abs{M}$\;
			}
			\Else{
				$M[j] = \VectorLength(p(A),m,w')$\;
			}
		}
		$B = [\;]$\;
		\For{$i \gets 1\dots k$}{
			Append \SpinUntil{$A,w_i,\lambda(w_i)\cdot d$} to end of $B$\;
		}
		\Return{$B$}\;
	}
\end{algorithm}

The subsequent iterative reduction requires $\On(n^2)$ operations per step and is executed at most $km-n$ times. Finally, spinning all of the resulting vectors costs $\On(n^3)$ per spin, performed at most $n$ times. 

Combining these estimates, the overall complexity remains $\On(n^4)$ field operations for decomposing a primary matrix $A \in \F^{n\times n}$ into cyclic blocks.

\section{Putting together the Jordan normal form}

We now turn to the computation of the individual Jordan blocks corresponding to each primary component. This is done by computing a basis of $V$ as described in \cref{Jordan Block}.

\begin{algorithm}[H]
	\caption{Compute Jordan block form}
	\KwData{$A \in \F^{n\times n}$ with $\mu_A = p^m$, $p$ irreducible}
	\KwResult{Matrix $B \in \GL{n}{\F}$ such that $BAB^{-1} = J(p^m)$}
	\SetKwFunction{JordanBlock}{JordanBlock}
	\Fn{\JordanBlock{$A,p,m$}}{
	$d \gets \deg(p)$\;
	$v,\dots, vA^{n-1} \gets$ \FindCyclicVector$(A)$\;
	$B \gets 0 \in \F^{n\times n}$\;
	Set first $d$ rows of $B$ to $v_1,\dots,vA^{d-1}$\; 
	\For{$r \gets 1$ \KwTo $d$}{
		\For{$i \gets 1$ \KwTo $m-1$}{
			$b \gets$ \EvaluatePolynomialWithSpun($A,p,[vA^{r},\dots,vA^{n-1}]$)\;
			Set row $d\cdot i + r$ of $B$ to $b$\;
		}
	}
	\Return{$B$}\;
	}
\end{algorithm}

The computation begins with finding a cyclic vector, which, as established earlier, requires $\On(n^3)$ field operations. Next, the polynomial $p$ is evaluated at $A$ a total of $n$ times. Since each evaluation is applied to a vector whose images under powers of $A$ are already known, we can employ the efficient procedure from \cref{nicepolyeval}, reducing the cost to $\On(n\cdot d)$ field operations per evaluation. 
Consequently, the overall complexity of computing the conjugacy matrix for the Jordan block of a primary, cyclic matrix in $A \in \F^{n\times n}$ is dominated by the cyclic vector search, resulting in a total of $\On(n^3)$ field operations.

Now computing the Jordan normal form of a matrix is only a matter of putting all of the pieces together. 

\begin{algorithm}[H]
	\caption{Jordan normal form}
	\KwData{$A \in \F^{n\times n}$}
	\KwResult{Matrix $B \in \GL{n}{\F}$ such that $BAB^{-1} = \mathrm{JNF}(A)$}
	\SetKwFunction{JordanNormalform}{JordanNormalform}
	\Fn{\JordanNormalform{$A$}}{
		$\{\{p_1,m_1\},\dots,\{p_{\ell},m_{\ell}\}\} \gets$ multiset of collected factors of $\mu_A$\;
		$PrimaryBasis \gets$ \PrimaryDecomposition{A}\;
		$\bigoplus_{i=1}^{\ell}P_i \gets A^{PrimaryBasis^{-1}}$\;
		\For{$i \in \{1,\dots,\ell\}$}{
			$CyclicBasis_i \gets$ \CyclicDecomposition{$P_j,p,m$}\;
			$\bigoplus_{j=1}^k C_j \gets P_i^{CyclicBasis_i}$\;
			\For{$j \in \{1,\dots,k\}$}{
				$JordanBasis_j \gets \JordanBlock(C_j)$\;
			}
			$CombinedBases_i \gets (\bigoplus_{j=1}^k JordanBasis_j) \cdot CyclicBasis_i$ 
		}
		\Return{($\bigoplus_{i=1}^{\ell}CombinedBases_i)\cdot PrimaryBasis$}
	}
\end{algorithm}

We see that computing the Jordan normal form for an arbitrary $A \in \F^{n\times n}$ depends very much on the structure of the matrix. Let $\ell$ denote the number of primary subspaces of $V$ and for each primary subspace let $k_i$ be the number of cyclic subspaces it decomposes into. Computing the primary decomposition costs $\On(n^4)$ field operations. Then computing the cyclic decomposition of the primary subspaces costs $\On(\frac{n}{d_i}^4)$ field operations each, where $d_i$ denotes the dimension of each primary subspace. Finally, we compute the Jordan block for each of the cyclic subspaces, so we do it $k_1 + \dots + k_{\ell} = \ell$ times. Thus, computing the Jordan blocks contained in each primary subspace costs $\On((\ell\cdot \frac{n}{d_i})^3) $ field operations. 
Consequently, computing the Jordan normal form requires $\On(\ell \cdot n^4)$ field operations.

We see that the complexity of computing the Jordan normal form depends on the sizes of the primary and cyclic subspaces arising in the decomposition of V.
Let $\ell$ be the number of primary subspaces $P_1,\dots,P_{\ell}$ with $d_i = \dim(P_i)$ (so $\sum_{i=1}^{\ell}d_i = n)$ and let $k_i$ be the number of cyclic subspaces $C_{i,1},\dots,C_{i,k_i}$ in the $i$-th primary subspace with $m_{i,j} := \dim(C_{i,j})$ (so $\sum_{j=1}^{\ell} m_{i,j} = d_i$). 

The primary decomposition of $A$ costs $\On(n^4)$ field operations. Subsequently, the cyclic decomposition of each primary block $P_i$ costs $\On(n_i^4)$ field operations. Finally, computing the Jordan block form of each cyclic block $C_{i,j}$ costs $\On(m_{i,j}^4)$ field operations. Thus the overall complexity is given by
$$\On(n^4)  + \sum_{i=1}^{\ell} \On(n_i^4) + \sum _{i=1}^{\ell}\sum_{j=1}^{k_i}\On(m_{i,j}^4).$$ 
Since both $\sum_{i=1}^{\ell} \On(n_i^4) \leq n^4$ and $\sum _{i=1}^{\ell}\sum_{j=1}^{k_i}\On(m_{i,j}^4) \leq n^4$, the worst case is dominated by $\On(n^4)$. However, the complexity of the computation can be significantly lower if the primary and cyclic blocks are small.

%% file: parts/03special.tex
\chapter{Special cases}

We now present two special cases of matrices, for which we can speed up the computation of the Jordan normal form significantly using theoretical considerations. Both of these special cases are easily checked by computing the factorised minimal polynomial of the matrix at the beginning. Again, for this chapter let $q$ be a power of a prime and $\F = \F_q$ be the field with $q$ elements. 

\section{Algorithms}

\subsection{Cyclic Matrices}
The first special case we will consider is that of cyclic matrices. It turns out that almost all matrices over finite fields are cyclic, and the proportion increases with increasing field size. The following result has been proven by Praeger and Neumann in \cite[p.270]{NeChe95}:
\begin{theorem}
	Let $A \in \F^{n\times n}$. Then the probability $p$ that $A$ is cyclic is given by 
	$$p > 1 - \frac{1}{(q^2-1)(q-1)}$$ 
\end{theorem}
So even for a field as small as $\F_5$, the probability of a random matrix $A \in \F_5^{n\times n}$ being cyclic is already greater than $0.98$. Furthermore, we can easily check the cyclicity of a matrix by checking whether the degree of its minimal polynomial is equal to its dimension. This warrants modifying the algorithm for the computation of the Jordan normal form to take advantage of the structure of cyclic matrices. 

Recall the algorithm for the primary decomposition of $V$ given a matrix $A \in F^{n\times n}$. It generates random vectors with primary minimal polynomials to get a set of generating vectors for each distinct irreducible factor of $\mu_A$. This can be done much more efficiently if it is known that the matrix is cyclic beforehand, thanks to the following proposition:

\begin{prop}\label{alrcy}
	Let $A \in \F^{n\times n}$ be cyclic and let $v$ be a cyclic vector. Furthermore, let $\mu_A = \prod_{i=1}^{\ell}p_i^{m(i)}$ be the factorisation of the minimal polynomial into distinct irreducible factors. Then 
	\begin{enumerate}
		\item $V_i := \kernel(p_i^{m(i)}(A))$ is cyclic for $i \in \{1,\dots,\ell\}$
		\item and a generating vector of $V_i$ is given by $vq_i(A)$, where $q_i := \prod_{\substack{j\neq 1 \\j=1}}^{\ell}p_j^{m(j)}$.
	\end{enumerate}
\end{prop}

\begin{proof}
	The first statement follows directly from \cref{cyclicsub} since $V_i$ is $A$-invariant.
	
	Now, since $v$ is a cyclic vector, we know that $\mu_{A,v} = \mu_A = \prod_{i=1}^{\ell}p_i^{m(i)}$. Now let $q_i$ be as described in the proposition. Then for $w_i := v\cdot q_i(A)$ we have $\mu_{A,w_i} = p_i^{m(i)}$. So $w_i \in V_i$ and moreover $\matspan{w_i}_A \leq V_i$, since the latter is $A$-invariant. 
	Because 
	\begin{align*}
		\sum_{i=1}^{\ell}\dim(\matspan{w_i}_A) &= \sum_{i=1}^{\ell}\deg(\mu_{A,w_i})\\
											   &= \sum_{i=1}^{\ell}\deg(p_i^{m(i)})\\
											   &= \deg(\mu_A) \\
											   &= n.
	\end{align*}							   
	 we have $\dim(\matspan{w_i}_A) = \kernel(p_i^{m(i)}(A))$, so the two subspaces are equal.
\end{proof}

This tells us that we only need to find a cyclic vector for $A$ to be able to directly construct all of the cyclic generating vectors for the primary subspaces. 

\begin{algorithm}[H]
	\caption{Primary decomposition of cyclic matrix}
	\KwData{$A \in \F^{n \times n}$ cyclic, $\mu_A = \prod_{i=1}^{\ell}p_i^{m(i)}$}
	\KwResult{A matrix $B$ such that $BAB^{-1}$ is in primary decomposition form}
	\SetKwFunction{PrimaryDecompositionCyclic}{PrimaryDecompositionCyclic}
	\SetKwProg{Fn}{Function}{:}{end}
	\Fn{\PrimaryDecompositionCyclic{$A,\mu_A$}}{
		$vs \gets \FindCyclicVector(A)$ \tcp*{\{$v,\dots,vA^{n-1}\}$ for a cyclic vector $v$}
		$B \gets [\;]$\;
		\For{$i \in \{1,\dots,\ell\}$}{
			$q_i \gets \mu_A/p_i^{m(i)}$\;
			$w \gets \EvaluatePolynomialWithSpun(A,q_i,vs)$\;
			$wspan \gets \SpinUntil(w,A,\deg(p_i^{m(i)}))$\;
			Append $wspan$ to the end of $B$\;
		}
		\Return{$B$}\;
	}
\end{algorithm}

This approach avoids testing multiple vectors and echelonising their $A$-spans entirely. Moreover, computing a cyclic vector $v$ directly yields the basis $\{v,vA,\dots,vA^{n-1}\}$ of $V$, enabling the use of \cref{nicepolyeval} for efficient polynomial evaluation at $A$. 

Let $\ell$ be the number of primary subspaces of $V$ for $A \in \F^{n\times n}$. Finding a cyclic vector costs $\On(n^3)$ field operations. In the main loop, for each $i$ let $d_i := \deg(p_i^{m(i)})$. We perform:
\begin{enumerate}
	\item a polynomial division costing $\On(n\cdot d_i)$ field operations,
	\item the evaluation of $v_iq_i(A)$, costing $\On(n\cdot\frac{n}{d_i})$ field operations,
	\item and a construction of a basis of $\matspan{w_i}_A$ using the spinning algorithm, costing $\On(n^2\cdot\frac{n}{d_i})$ field operations. 
\end{enumerate}
So overall the number of field operations taken by the main loop is bound by $\On(\ell\cdot n^3)$. Consequently, computing the primary decomposition of a cyclic matrix only costs $\On(\ell\cdot n^3)$ field operations. 

Since it also follows from \cref{alrcy} each primary subspace of $V$ is already $A$-cyclic, no further cyclic decomposition is required, and the Jordan block form for each can be computed directly. For a primary block of size $d_i$, the computation of the Jordan block form costs $\On(d_i)$ field operations. Since $\sum_{i=1}^{\ell} d_i = n$ and there are $\ell$ such blocks the cost of computing all Jordan blocks is $\On(n^3)$ field operations. So computing the Jordan normal form of a cyclic matrix overall costs $\On(\ell\cdot n^3)$ field operations, where $\ell \leq n$. 

\subsection{Matrices with irreducible minimal polynomial}

We now consider another special case of matrices that can be easily identified by the factorisation of their minimal polynomials. 

\begin{bem}
	Let $A \in \F^{n\times n}$ with $\mu_A = p$ irreducible. Then its characteristic polynomial is equal to $p^c$ for some $c\in \mathbb{N}$ and the Jordan normal form of $A$ is of the following form:
	$$B = I_c \otimes M_p := 
	\begin{pNiceArray}{cccc}
		M_p& & & \\
		&M_p& & \\
		& &\ddots& \\
		& & &M_p 
	\end{pNiceArray}
	\in \mathbb{F}^{cs \times cs}$$
	where $I_c$ denotes the identity matrix of dimension $c$ and $\otimes$ denotes the tensor product. 
\end{bem}

If $\mu_A$ is irreducible, the following special case of \cref{directsumlemma} gives us a very powerful result for computing their Jordan normal form.

\begin{lem} \label{directsumlemmaspecial}
	Let $A \in \F^{n\times n}$ with $\mu_A = p$ for some irreducible $p \in \F[X]$. Furthermore, let $v, w \in V\backslash\{0\}$. It then follows that:
	$$\langle v \rangle_A + \langle w \rangle_A = \langle v \rangle_A \oplus \langle w \rangle_A$$
	if and only if $\langle v \rangle_A \neq \langle w \rangle_A$.
\end{lem}

\begin{proof}
	Since $\mu_{A,v}$ and $\mu_{A,w}$ divide $p$ and $\mu_{A}$ is irreducible, it follows that $\mu_{A,v} = \mu_{A,w} = p$. Because both $A|_{\matspan{v}_A}$ and $A|_{\matspan{w}_A}$ are cyclic, the only non-trivial $A$-invariant subspaces of $\matspan{v}_A$ and $\matspan{w}_A$ are $\matspan{v}_A$ and $\matspan{w}_A$ respectively.
	
	If $\matspan{v}_A = \matspan{w}_A$, then $\matspan{v}_A + \matspan{w}_A$ is not a direct sum. Let $\matspan{v}_A \neq \matspan{w}_A$. As $\matspan{v}_A \cap \matspan{w}_A$ is an $A$-invariant subspace of both $\matspan{v}_A$ and $\matspan{w}_A$, we have $\matspan{v}_A \cap \matspan{w}_A = \{0\}$. Therefore, $\matspan{v}_A + \matspan{w}_A = \matspan{v}_A \oplus \matspan{w}_A$ 
\end{proof}

To compute the Jordan normal form of a matrix $A$ whose minimal polynomial $\mu_A = p$ is irreducible one can proceed as follows. Select a random nonzero vector $v \in \F^{n\times n}$. Since $p$ is irreducible, the minimal polynomial of $v$ is also $p$, so $\dim(\matspan{v}_A) = \deg(p)$. A basis for $\matspan{v}_A$ can then be obtained by applying the spinning algorithm until $\deg(p)$. Next, repeatedly choose new vectors that do not already lie in the direct sum of the $A$-spans of the vectors already recorded. By \cref{directsumlemmaspecial}, the $A$-span of $w$ intersects this sum nontrivially if and only if $w$ already belongs to it. Thus, selecting $w$ outside of the current sum yields a new $A$-invariant, cyclic subspace of dimension $\deg(p)$. Iterating this process produces vectors $v_1,\dots,v_k$ such that $V = \matspan{v_1}_A\oplus\dots \oplus\matspan{v_k}_A$. 

Since the Jordan Blocks of $A$ are all of the form $M_p$ the conjugacy matrix computed so far is already the conjugacy matrix for the Jordan normal form. 

\begin{algorithm}[H]
	\caption{Jordan Normal form for matrices with irreducible minimal polynomials}
	\KwData{$A \in \F^{n \times n}$ with $\mu_A$ irreducible}
	\KwResult{$B\in \GL{n}{\F}$ such that $BAB^{-1} = \mathrm{JNF}(A)$}
	\SetKwFunction{JordanNormalFormIrred}{JordanNormalFormIrred}
	\SetKwProg{Fn}{Function}{:}{end}
	\Fn{\JordanNormalFormIrred{$A$}}{
	$blocksize := \deg(\mu_A)$\;
	$v \gets$ random vector in $\F^n\backslash{0}$ \;
	$\mathrm{V_1} \gets \SpinUntil(v, A, blocksize)$\; 
	k := 1\;
	\For{$i \in \{1,\dots,\frac{n}{blocksize}-1\}$}{
		$w \gets$ random vector in $\F^n\backslash \sum_{i=1}^k\matspan{V_i}$\;
		$V_i \gets \SpinUntil(w, A, blocksize)$\;
		$k \gets k+1$\;
	}
	$B \gets$ matrix with $V_1,\dots,V_k$ as its rows\;
	\Return{$B$}\;
	}
\end{algorithm}

Let $d := \deg(\mu_A)$. The overall complexity is determined by the spinning algorithm which is invoked $\frac{n}{d}$ times. With each invocation, the vector is spun exactly $d$ times, resulting in a cost of $\On(d\cdot n^2)$ field operations. So overall, computing the Jordan normal form for $A$ costs $\On(\frac{n}{d} \cdot d\cdot  n^2) = \On(n^3)$ field operations.

\section{Practical Comparisons}

In this section, we present a series of runtime experiments comparing the different algorithms for computing the Jordan normal form over finite fields. This is done to illustrate their practical performance depending on both the input field and matrix size as well as the tangible advantage of specialized methods over naive implementations. The results are summarized in the following figures. 

All experiments were carried out using the computer algebra system GAP (version 4.14.0). The computations were performed on a machine equipped with an AMD Ryzen 5 3600 six-core processor (3.60GHz), 32 GB RAM, running Windows 10. The implementation was done in GAP, relying on Meinolf Geck's unpublished \texttt{NoFoMa} package \cite{Gec23}. These details are included to ensure reproducibility and to provide context for the reported runtimes.  

\begin{figure}[h]
	\includegraphics[scale = 0.1]{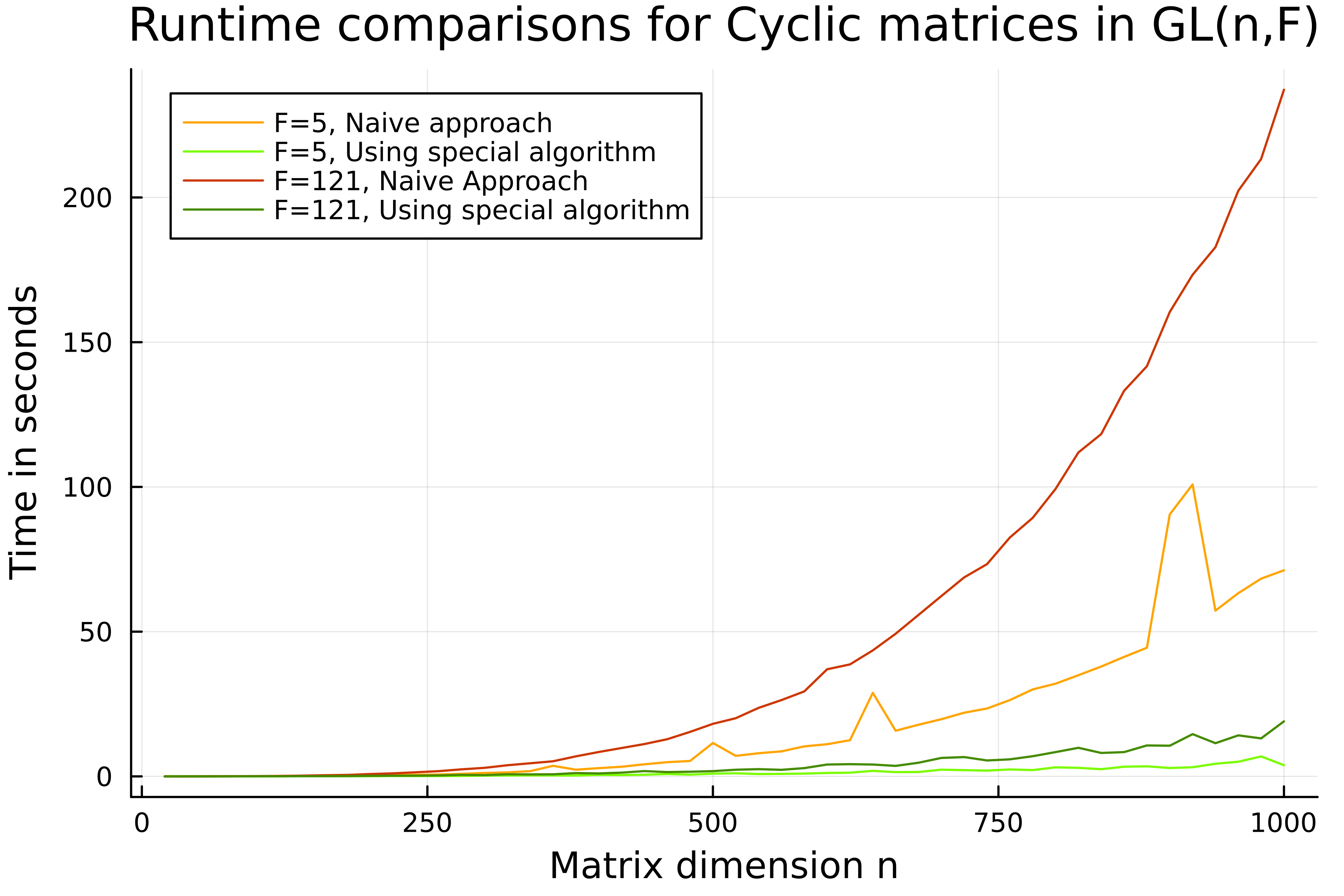}
	\centering
\end{figure}

For matrix dimensions ranging from $n=20$ to $n=1000$, in steps of $20$, a random cyclic matrix over $\F_5$ was generated. For each such matrix, the Jordan normal form was computed twice: once using the naive algorithm and once using the specialised algorithm for cyclic matrices. The same procedure was then carried out for matrices over $\F_{11^2}$. The measured runtimes are shown in the figure above. In the cases of larger dimensions and field size, the specialized method achieves speedups of more than a factor of twenty compared to the naive approach. 

\begin{figure}[h]
	\includegraphics[scale = 0.1]{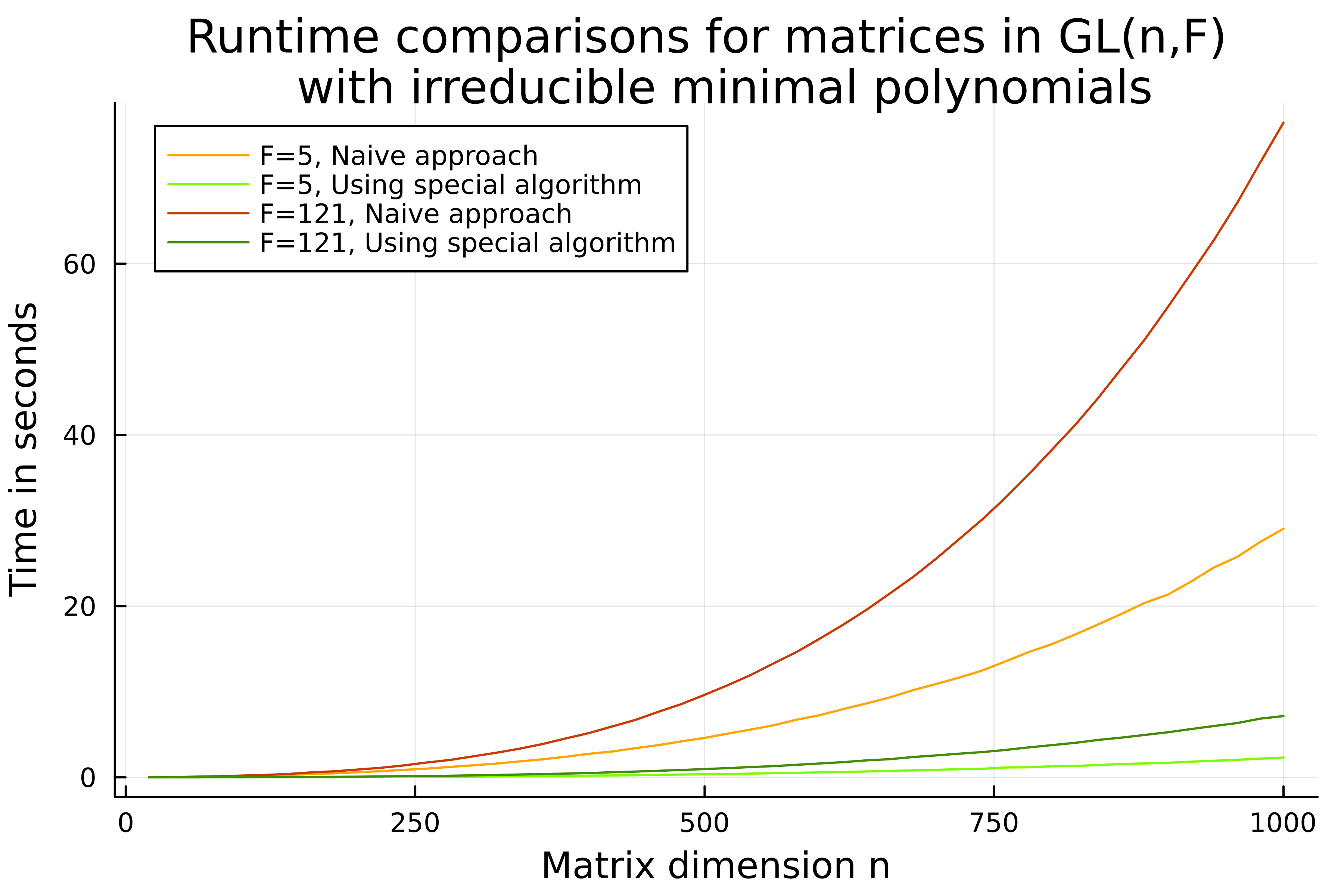}
	\centering
\end{figure}

The same procedure was repeated with matrices whose minimal polynomial is irreducible. In this setting, the specialized algorithm still outperforms the naive one, yielding speedups of a factor of 10 for larger matrix sizes and the larger field. 

These experiments illustrate the practical differences in runtime between the naive and specialized algorithms. In particular, the specialized approaches consistently provide significant performance improvements, especially as the dimension and field size grow.

%% file: parts/05conj.tex
\chapter{Conjugacy in the classical groups}
We have shown that two matrices are similar if and only if their Jordan normal forms coincide. Since the conjugating matrix always lies in $\GL{n}{\F}$ by definition, two matrices $A,B \in \GL{n}{\F_q}$ are conjugate in $\GL{n}{\F_q}$ if and only if their Jordan normal forms coincide. However, the following question arises: in which other groups is having the same Jordan normal form equivalent to conjugacy in the group? And which additional requirements do we need for conjugacy in groups in which the equivalence does not hold? 

This question has been thoroughly discussed in \cite{Fra20} and \cite{Fra25}, we will summarise the results in this chapter without proof. 

\section{Finite Classical Groups}

We begin by introducing the classical groups of finite fields, following the definitions used in \cite{Gro01}. First we give a few definitions, which we wil need later to define the classical groups.

\begin{defi}[Forms]
	Let $V$ be an $\F$-vector space. 
	\begin{enumerate}
		\item We call a map $\beta : V\times V \rightarrow \F, (v,w) \mapsto \beta(v,w)$ a \textbf{bilinear form} if $\beta$ is linear in each component, that is if for all $u,v,w \in V$ and $a,b \in \F$:
		\begin{itemize}
			\item $\beta(au+bv,w)=a\beta(u,w) + b\beta(v,w)$ and
			\item $\beta(u,av+bw)=a\beta(u,v) + b\beta(u,w)$.
		\end{itemize}
		\item We call a map $Q: V \rightarrow \F$ a \textbf{quadratic form} if there is an associated bilinear form $\beta_Q$ such that for all $v,w \in V$ and $a \in \F$:
		\begin{itemize}
			\item $Q(av) = a^2Q(v)$ and
			\item $\beta_Q(v,w) = Q(v+w) - Q(v) - Q(w)$.
		\end{itemize}
		$\beta_Q$ is called the \textbf{polar form of $Q$}. 
		\item If $\F$ is a field admitting a field automorphism $\overline{(\cdot)}$ of order 2 we call a map $\eta: V\times V \rightarrow \F, (v,w) \mapsto \eta(v,w)$ a \textbf{hermitian form} if for all $u,v,w \in V$ and $a,b \in \F$:
		\begin{itemize}
			\item $\eta(au+bv, w) = a\eta(u,w) + b\eta(v,w)$ and 
			\item $\eta(v,w) = \overline{\eta(w,v)}$.
		\end{itemize}
	\end{enumerate}
	From now on we will refer to a bilinear form or hermitian form only as a \textbf{form}. 
\end{defi}

We define the following properties of forms: 

\begin{defi}
	Let $V$ be an $\F$-vectorspace equipped with either a bilinear form $\beta$ or a quadratic form $Q$. 
	\begin{enumerate}
		\item $\beta$ is \textbf{non-degenerate} if for every $v \in V\backslash \{0\}$ there exists a $w \in V$ such that $\beta(v,w) \neq 0$. Otherwise, it is \textbf{degenerate}. 
		
		$Q$ is \textbf{non-degenerate} if the polar form of $Q$ is non-degenerate. Otherwise $Q$ is \textbf{degenerate}. 
		\item $\beta$ is \textbf{symmetric} if $\beta(v,w) = \beta(w,v)$ for all $v,w \in V$. 
		\item $\beta$ is \textbf{skew-symmetric} if $\beta(v,w) = -\beta(w,v)$ for all $v,w \in V$.
		\item $\beta$ is \textbf{alternating} if $\beta(v,v) = 0$ for all $v \in V$. 
		\item A non-degenerate and alternating bilinear form is \textbf{symplectic}.
		\item A non-degenerate and Hermitian form is a \textbf{unitary form}. 
		\item If $\mathrm{Char}(\F) \neq 2$, then a non-degenerate and symmetric bilinear form is \textbf{orthogonal}.
	\end{enumerate}
\end{defi}

We define a special type of endomorphism, which is preserved by forms. 

\begin{defi}
	Let $V$ be an $n$-dimensional $\F$-vector space equipped with a form $\beta$ or a quadratic form $Q$. A map $X \in \GL{n}{\F}$ is an \textbf{isometry} of $\beta$ if $\beta(v^X,w^X) = \beta(v,w)$ for all $v \in V$ or $Q(v^X) = Q(v)$ for all $v \in V$. The set of isometries of $\beta$ is a subgroup of $\GL{n}{\F}$ which we will denote by $\mathcal{C}(\beta)$. 
\end{defi}

We can now define classical groups based on forms and isometries:

\begin{defi}
	Let $V$ be an $n$-dimensional $\F$ vector space equipped with a non-degenerate form $\beta$ or a non-degenerate quadratic form $Q$. 
	\begin{enumerate}
		\item The group of linear isomorphisms of $V$ is the \textbf{linear group} and it is denoted by $\GL{n}{\F}$. 
		
		The \textbf{special linear group} $\SL{n}{\F}$ is the subset of all elements of $\GL{n}{\F}$ having determinant $1$. 
		\item If $\beta$ is symplectic, $\mathcal{C}(\beta)$ is the \textbf{symplectic group} and denoted by $\Sp{n}{\F}$. 
		\item If $\beta$ is unitary, $\mathcal{C}(\beta)$ is the \textbf{unitary group} denoted by $\mathrm{U}(n,\F)$.
		
		The \textbf{special unitary group} $\SU{n}{\F}$ is the subset of all elements of $\mathrm{U}(n,\F)$ having determinant $1$. 
		
		\item $\mathcal{C}(Q)$ is the \textbf{orthogonal group} and is denoted by $\mathrm{O}(n,\F)$.
		
		The \textbf{special orthogonal group} $\SO{n}{\F}$ is the subset of all elements of $\mathrm{O}(n,\F)$ having determinant $1$. 
		
		The \textbf{Omega group} is the derived subgroup of $\SO{n}{\F}$ and we denote it by $\Omega(n,\F)$.
	\end{enumerate}
\end{defi}

We will consider three different cases for the orthogonal groups, with the Witt index as the distinguishing feature. To define the Witt index, we present a few additional properties of forms. 

\begin{defi}
	Let $V$ be a vector space over $\F$ and let $\beta$ be a hermitian form on $V$. Then:
	\begin{enumerate}
		\item We call $v \in V\backslash\{0\}$ \textbf{isotropic} if $\beta(v,v) = 0$ and a subspace $W \leq V$ is \textbf{totally isotropic} if $W \subset W^{\perp} = \{v \in V\mid\beta(v,w) = 0 \text{ for all } w \in W\}$. 
		\item We call $v \in V\backslash\{0\}$ \textbf{singular} if $Q(v) = 0$ and a subspace $W \leq V$ \textbf{totally singular} if $Q(u) = 0$ for alle $u \in W$. 
		\item A subspace $W \leq V$ is \textbf{non-degenerate} if $W \cap W^{\perp} = \{0\}$.
		\item If $V = U \oplus W$ and $\beta(u,w) = 0$ for all $u\in U$ and $w \in W$, we write $U \perp W$ and say that $V$ is the \textbf{orthogonal direct sum of $U$ and $W$}. 
	\end{enumerate}
\end{defi}

\begin{defi}
	Let $V$ be an $\F$-vector space with a non-degenerate symplectic or unitary form $\beta$ or a non-degenerate quadratic form $Q$ in which case $\beta$ denotes the polar form of $Q$. We call a pair of vectors $(v,w) \in V^2$ a \textbf{hyperbolic pair} if $v,w$ are both isotropic and $\dim(\matspan{v,w})=2$. The subspace $\matspan{v,w}$ spanned by a hyperbolic pair $(v,w)$ is called a \textbf{hyperbolic plane}. 
\end{defi}

\begin{defi}
	Let $V$ be an $\F$-vector space with a non-degenerate symplectic or unitary form $\beta$ or a quadratic form $Q$ in which case $\beta$ is the polar form of $Q$. The dimension of a maximally totally isotropic subspace of $V$ or, in case of a quadratic form, a maximally totally singular subspace of $V$, is called the \textbf{Witt index} of $(V,\beta)$ or $(V,Q)$.  
\end{defi}

\begin{theorem}
	Let $\beta$ be a symplectic bilinear form on an $\F$-vector space $V$ with $\dim(V) \geq 2$. Then $V$ admits a basis $\mathcal{B} = (e_1,\dots,e_m,f_m,\dots,f_1)$ where $\erz{e_i, f_i}$ are hyperbolic planes for all $i \in \{1, \dots, m\}$ such that 
	$$V = \erz{e_1, f_1}  \perp \dots \perp \erz{e_m, f_m}$$
	and $V$ has Witt index $m$. 
\end{theorem}

\begin{theorem}
	Let $\beta$ be a unitary form on an $n$-dimensional $\F$-vector space $V$ with $n \geq 2$. There exist vectors $e_1,\dots,e_m,f_1,\dots,f_m \in V$ such that
	$$V = \erz{e_1,f_1} \perp \dots \perp \erz{e_m,f_m} \perp V_0$$ 
	where $\left\langle e_i,f_i \right\rangle$ are hyperbolic planes for all $1 \leq i \leq m$ and $V_0 \leq V$ is anisotropic, such that $V$ has Witt index $m$. One of the following two cases holds:
	\begin{enumerate}
		\item $\dim(V_0) = 0$ and $\dim(V) = 2m$ or 
		\item $\dim(V_0) = 1$ and $\dim(V) = 2m+1$. For $V_0 = \erz{w}$ we can assume $\beta(w,w)=1$.  
	\end{enumerate}
\end{theorem}

\begin{theorem}\label{orthogonaldecomp}
	Let $\mathrm{Char}(\F)$ be odd and let $\beta$ be an orthogonal form on an $\F$-vector space $V$. Then there exist vectors $e_1,\dots,e_m,f_1,\dots,f_m \in V$ such that 
	$$V = \erz{e_1,f_1} \perp \dots \erz{e_m,f_m} \perp V_0$$
	where $\erz{e_i,f_i}$ is a hyperbolic plane and $V_0 \leq V$ is anisotropic, such that $V$ has Witt index $m$ and exactly one of the following holds: 
	\begin{enumerate}
		\item $\dim(V_0) = 0$ and $\dim(V)=2m$.
		\item $\dim(V_0) = 1$ and $\dim(V)=2m+1$. For $V_0 = \erz{w}$ we can assume that \mbox{$\beta(w,w) = -2^{-1}$}. 
		\item $\dim(V_0) = 2$ and $\dim(V)=2m+2$. For $V_0 = \erz{w_1,w_2}$ we can assume that $\beta(w_1,w_1) = -2, \beta(w_1,w_2)= 0$ and $\beta(w_2,w_2)=2\omega$ where $\omega$ is a primitive element of $\F$.  
	\end{enumerate} 
\end{theorem}

Based on this result, we emphasise for the orthogonal groups these distinctions in the naming convention. 

\begin{defi}
	Let $\mathrm{Char}(\F)$ be odd and $\beta$ be an orthogonal form on an $n$-dimensional $\F$-vector space $V$. Let 
	$$V = \matspan{v_1,w_1} \perp \dots \perp \matspan{v_m,w_m} \perp V_0$$ 
	as in \ref{orthogonaldecomp}. We define the following types of orthogonal groups:
	\begin{enumerate}
		\item If $\dim{V_0} = 0$ then we call $\beta$ hyperbolic and $\mathcal{C}(\beta)$ the \textbf{orthogonal group of plus type}, denoting it by $\mathrm{O}^+(n,\F)$. The \textbf{special orthogonal group of plus type} $\SO{+}{n}{\F}$ is the subset of all elements of $\mathrm{O}^+(n,\F)$ having determinant $1$. 
		
		\item If $\dim{V_0} = 1$ then we call $\beta$ parabolic and $\mathcal{C}(\beta)$ the \textbf{orthogonal group of circle type}, denoting it by $\mathrm{O}^{\circ}(n,\F)$. The \textbf{special orthogonal group of circle type} $\SO{\circ}{n}{\F}$ is the subset of all elements of $\mathrm{O}^{\circ}(n,\F)$ having determinant $1$. 
		
		\item If $\dim{V_0} = 2$ then we call $\beta$ elliptic and $\mathcal{C}(\beta)$ the \textbf{orthogonal group of minus type}, denoting it by $\mathrm{O}^-(n,\F)$. The \textbf{special orthogonal group of minus type} $\SO{-}{n}{\F}$ is the subset of all elements of $\mathrm{O}^-(n,\F)$ having determinant $1$. 
	\end{enumerate}
\end{defi}

\section{The linear groups}
From here on let $q$ be a prime power and let $\F = \F_q$ be the field with $q$ elements.

We have already shown in \cref{JNFequiv} that two matrices $A,B \in \GL{n}{\F}$ are conjugate in $\GL{n}{\F}$ if and only if they have the same Jordan normal form. So given a conjugacy class of $\GL{n}{\F}$, we may fix it as the representative. 

\begin{theorem}[{Conjugacy in $\SL{n}{\F}$}]\label{slconj}
Let $A$ in $\SL{n}{\F}$ and $d:=\gcd(\lambda_1,\dots,\lambda_k,q-1)$ where $\lambda_i$ denotes the dimension of the $i$-th Jordan block of $A$ for $1 \leq i \leq k$. Furthermore let $\omega$ be a primitive element of $\F$ and $D \in \GL{n}{\F}$ with determinant $\omega$. 

Then the conjugacy class of $A$ in $\GL{n}{\F}$ splits into $d$ distinct classes in $\SL{n}{\F}$ with representatives $A, A^D, A^{D^2}, \dots, A^{D^{d-1}}$. Moreover for $C_1,C_2 \in \GL{n}{\F}$, $A^{C_1}$ and $A^{C_2}$ are conjugate in $\SL{n}{\F_1}$ if, and only if, $\det(C_1^{-1}C_2)$ is a power of $\omega^d$. 
\end{theorem}

Let $d$ and $\omega$ be as described in \cref{slconj}.
Then given matrices $A, B \in \SL{n}{\F}$ and $X_A, X_B \in \GL{n}{\F}$ such that 
$$A^{X_A} = \JNF(A) = \JNF(B) = B^{X_B}$$ 
we get that $A = A^{I}$ and $B = A^{X_A\cdot X_B^{-1}}$ are conjugate in $\SL{n}{\F}$ if, and only if, $\det(I^{-1}X_AX_B^{-1}) = \det(X_AX_B^{-1}) = \omega^s$ for some $s \in \mathbb{N}_0$.  

\section{The finite classical groups}
In this section let $\epsilon\in \{+,-,\circ\}$, with $n$ odd if $\epsilon = \circ$ and even if $\epsilon \in \{+,-\}$ or $G = \Sp{n}{\F}$. Furthermore, let $G$ be a finite classical group.

We first look at the special case of semisimple elements, i.e. elements with minimal polynomial that is a product of distinct irreducible factors. 

\begin{theorem}[Conjugacy of semisimple elements]
	Let $A, B \in G$ be semisimple. 
	\begin{enumerate}
		\item If $G= \Sp{n}{\F}$ or $G = \U{n}{\F}$ then $A$ and $B$ are conjugate in $G$ if and only if their Jordan normal forms coincide. 
		\item For $G = \SU{n}{\F}$, the conjugacy class of $A \in G$ coincides with the conjugacy class of $A$ in $\U{n}{\F}$. So $A$ and $B$ are conjugate if, and only if, their Jordan normal forms coincide.
	\end{enumerate}
	\begin{enumerate}\setcounter{enumi}{2}
		\item Let $G = \mathrm{O}^{\epsilon}(n,\F)$. 
		\begin{itemize}
			\item If $\epsilon = \circ$, and $\beta$ a form inducing $G = C(\beta) = \mathrm{O}^{\circ}(n,\F)$, then $A$ and $B$ are conjugate in $G$ if, and only if, their Jordan normal forms coincide and the forms induced by $\beta$ on the eigenspaces of the eigenvalues $+1$ and $-1$ have the same type. 
			\item If $\epsilon\in \{+,-\}$, $A$ and $B$ are conjugate if and only if their Jordan normal forms coincide. 
		\end{itemize}
		\item Let $G = \SO{\epsilon}{n}{\F}$. 
		\begin{itemize}
			\item If $A^2 - 1$ is singular, i.e. $\det(A^2-1) = 0$, then $A$ and $B$ are conjugate in $G$ if and only if they are conjugate in $\mathrm{O}^{\epsilon}(n,\F)$. 
			\item Else, if $A^2 - 1$ is not singular, the conjugacy class of $A$ in $\mathrm{O}^{\epsilon}(n,\F)$ splits into two distinct classes with representatives $A$ and $A^D$ for $D \in \mathrm{O}^{\epsilon}(n,\F)$ with $\det(D) = -1$.
		\end{itemize}
		\item Let $G = \Omega^{\epsilon}(n,\F)$ with $n > 2$. 
		\begin{itemize}
			\item If $\epsilon = \circ$ or if $\epsilon \in \{+,-\}$ and $A+1$ is singular, $A$ and $B$ are conjugate in $G$ if and only if they are conjugate in $\SO{\circ}{n}{\F}$.
			\item If $\epsilon \in \{+,-\}$, and $A+1$ is not singular, the conjugacy class of $A$ in $\SO{\epsilon}{n}{\F}$ splits into two distinct classes with representatives $A$ and $A^D$ for $D \in \SO{\epsilon}{n}{\F}\backslash\Omega^{\epsilon}(n,\F)$.
		\end{itemize}
	\end{enumerate}
\end{theorem}

We now consider the conjugacy of unipotent elements, i.e. elements $A$ such that there exists a $k \in \mathbb{N}$ with $(A-1)^k = 0$. These results have been shown in \cite{Fra25}, and we will summarise the most straight forward results here without proof.

\begin{theorem}[Conjugacy of unipotent elements]
	Let $A, B \in G$ be unipotent. 
	\begin{enumerate}
		\item If $G = \U{n}{\F}$, $A$ and $B$ are conjugate in $G$ if and only if their Jordan normal forms coincide. 
		\item For $G = \SU{n}{\F}$, $A$ and $B$ are conjugate in $G$ if and only if they are conjugate in $\SL{n}{\F}$. 
	\end{enumerate}
	Now let $q$ be odd. 
	\begin{enumerate}\setcounter{enumi}{2}
		\item Let $G = \Sp{n}{\F}$. If $r$ is the number of blocks of even dimension in $\mathrm{JNF}(A)$, the conjugacy class of $A$ in $\GL{n}{\F}$ splits into $2^r$ distinct conjugacy classes in $G$. 
		\item Let $G = O^{\epsilon}$ and let $s$ be the number of blocks of odd dimension in $\mathrm{JNF}(A)$. 
		\begin{itemize}
			\item If $s>0$, the conjugacy class of $A$ in $\GL{n}{\F}$ splits into $2^{s-1}$ distinct conjugacy classes in $G$.
			\item There are no elements in $G$ with $s=0$ for $\epsilon \in \{\circ, -\}$. For $\epsilon = +$ and $s=0$, $A$ and $B$ are conjugate if and only if their Jordan normal forms coincide.
		\end{itemize}
		\item Let $G = \SO{\epsilon}{n}{\F}$ and let $s$ be the number of blocks of odd dimension in $\mathrm{JNF}(A)$.
		\begin{itemize}
			\item If $s>0$, $A$ and $B$ are conjugate in $G$ if and only if they are conjugate in $\mathrm{O}(n,\F)$. 
			\item if $s=0$, the conjugacy class of $A$ in $\mathrm{O}(n,\F)$ splits into two distinct classes in $G$. 
		\end{itemize}
		\item Let $G = \Omega^{\epsilon}(n,\F)$. If the number of blocks of odd dimension in $\mathrm{JNF}(A)$ is equal to zero, the conjugacy class of $A$ in $ \SO{\epsilon}{n}{\F}$ splits into two distinct classes in $G$.
	\end{enumerate}
\end{theorem}

Any element $A \in G$ can be written uniquely as a product $A = SU = US$, where $S$ is semisimple and $U$ is unipotent. Then for $C \in G$ we have $C^{-1}AC = C^{-1}SUC = C^{-1}SCC^{-1}C$, with $C^{-1}SC$ and $C^{-1}C$ being unipotent and semisimple respectively. So if $A$ and $B$ are conjugate in $G$, so are their semisimple and unipotent parts. We have seen that the conjugacy classes of these two classes of elements are well studied, so the approach taken in \cite{Fra20} is to list all semisimple conjugacy classes of $G$ and for each representative list all classes having that fixed semisimple part.

We summarise the most straightforward results.

\begin{theorem}[Conjugacy of general elements]
	Let $A, B \in G$.
	\begin{enumerate}
		\item Let $G = \U{n}{\F}$, $C_U := C_{\U{n}{\F}}, C_{SU} := C_{\SU{n}{\F}}$ be the centralizers of the unitary group and the special unitary group respectively and $I := \abs{C_U : C_{SU}}$ be their index. Furthermore let $p_1^{m_1},\dots p_k^{m_k}$ be the elementary divisors of $A$. Then the conjugacy class of $A$ in $\GL{n}{\F}$ splits into 
		$$r := \frac{q+1}{I} = \gcd(m_1,\dots,m_k,q+1)$$ 
		distinct conjugacy classes. 
		\item For $G = \mathrm{O}^{\epsilon}(n,\F)$ and $q$ odd, the conjugacy class of $A$ in $\GL{n}{\F}$ splits into two distinct classes in $G$ if and only if $A$ has no elementary divisors $(X\pm 1)^m$ with $m$ odd. 
		\item Let $G = \Omega^{\epsilon}(n,\F)$ and $q$ be odd. Furthermore, let $A, B$ be conjugate in $\SO{\epsilon}{n}{\F}$, i.e. $A^C = B$ for some $C \in \SO{\epsilon}{n}{\F}$. Then $C$ lies in $\Omega^{\epsilon}(n,\F)$ if and only if for the decomposition $C = US$, as described before, $S$ lies in $\Omega^{\epsilon}(n,\F)$.
	\end{enumerate} 
\end{theorem}

We have seen that given two matrices of a finite classical group, one can check whether there exists a conjugating element in that group by computing the Jordan normal form and verifying a few extra requirements depending on the exact group. However, that does not guarantee that the conjugating matrix computed using the algorithm for the Jordan normal form lies in that group. The latter is, in general, a more difficult problem. In the future, it may be worth investigating the development of algorithms to compute such conjugating elements in their respective finite classical group.

%% file: literature.bib
@article{NeChe95,
	author = {Neumann, Peter and Praeger, Cheryl},
	year = {1995},
	month = {10},
	pages = {},
	title = {Cyclic Matrices Over Finite Fields},
	volume = {52},
	journal = {Journal of the London Mathematical Society. Second Series},
	doi = {10.1112/jlms/52.2.263}
}

@book{Fra25,
	AUTHOR = "Giovanni De Franceschi and Martin W. Liebeck and Eamonn A. O'Brien",
	TITLE = "Conjugacy in Finite Classical Groups",
	YEAR = "2025",
	PUBLISHER = "Springer Cham",
	DOI = "https://doi.org/10.1007/978-3-031-86461-2",
	EDITION = "1",
}

@misc{Fra20,
	title={Centralizers and conjugacy classes in finite classical groups}, 
	author={Giovanni De Franceschi},
	year={2020},
	doi = {10.48550/arXiv.2008.12651}, 
}

@misc{Gec23,
	author = {Meinolf Geck},
	title = {NoFoMa},
	year = {2023},
	publisher = {GitHub},
	journal = {GitHub repository},
	howpublished = {\url{https://github.com/geckmf/NoFoMa}},
	commit = {2aef3beeb5a75aae4220f777be0e640163b23e2c}
}

@book{Har70,
	AUTHOR = "Brian Hartley and Trevor O. Hawkes",
	TITLE = "Rings, Modules and Linear Algebra",
	YEAR = "1970",
	PUBLISHER = "Chapman and Hall",
}

@book{Gro01,
	title={Classical Groups and Geometric Algebra},
	author={Larry C. Grove},
	year={2001},
}

@article{Par84,
	author = "R. A. Parker",
	title = "The computer calculation of modular characters (the Meat-Axe)",
	journaltitle = "Computational Group Theory",
	year = "1984",
	pages = "267-274",
	publisher = "Academic Press",
}

@book{Hof71,
	author = {Hoffman, Kenneth and Kunze, Ray A.},
	title = {Linear Algebra},
	year = {1971},
	publisher = {Prentice Hall Inc.},
}

@book{Hal78,
	author = {Paul R. Halmos},
	title = {Finite-Dimensional Vector Spaces},
	year = {1978},
	publisher = {Springer-Verlag},
}

@article{St97,
	title = {A New Algorithm for the Computation of Canonical Forms of Matrices over Fields},
	journal = {Journal of Symbolic Computation},
	volume = {24},
	number = {3},
	pages = {409-432},
	year = {1997},
	url = {https://doi.org/10.1006/jsco.1996.0142},
	author = {Allan Steel},
}

@unpublished{Ple08,
	title = {Lineare Algebra II},
	year = {2008},
	author = {Wilhelm Plesken},
    note   = {Lecture notes, RWTH Aachen University},
}
